\documentclass[sn-mathphys]{sn-jnl}

\usepackage{amssymb,amsmath,mathrsfs, amsthm}
\usepackage{xcolor}

\jyear{2021}%

\theoremstyle{thmstyleone}%
\numberwithin{equation}{section}
\newtheorem{theorem}{Theorem}[section]
\newtheorem{proposition}[theorem]{Proposition}%
\newtheorem{corollary}{Corollary}[section]
\theoremstyle{thmstyletwo}%
\newtheorem{example}{Example}%
\newtheorem{remark}[theorem]{Remark}%

\theoremstyle{thmstylethree}%
\newtheorem{definition}[theorem]{Definition}%

\begin{document}

\title[Article Title]{Curvature of left-invariant complex Finsler metric on Lie groups
}


\author{\fnm{Kuankuan} \sur{Luo$^1$}}\email{luokk@stu.xmu.edu.cn}
\author{\fnm{Wei} \sur{Xiao$^2$}}\email{wxiaomath@126.com}
\author{\fnm{Chunping} \sur{Zhong$^3$}}\email{zcp@xmu.edu.cn}


\affil[1,3]{\orgdiv{School of Mathematical Sciences}, \orgname{Xiamen
University},  \city{Xiamen}, \postcode{361005}, \country{P. R. China}}

\affil[2]{\orgdiv{School of Mathematics and Statistics}, \orgname{Minnan Normal University}, \city{Zhangzhou}, \postcode{363000},  \country{P. R. China}}



\abstract{
Let $G$ be an even-dimensional connected real Lie group admitting a left-invariant complex structure $J$. Let $\mathfrak g$ be the Lie algebra of $G$, and $\mathfrak g^{1,0}$ the $(1,0)$-part of the complexification $\mathfrak g^{\mathbb{C}}$ of $\mathfrak g$. We derive explicit formulas for the holomorphic sectional curvature and holomorphic bisectional curvature of any left-invariant strongly pseudoconvex complex Finsler metric $F$ on $G$, expressed in terms of the complex Lie algebra $\mathfrak g^{1,0}$. We also obtain a necessary and sufficient condition for $F$ to be a K\"ahler-Finsler metric (resp. weakly K\"ahler-Finsler metric). As an application, we prove a rigidity result: if $G$ is a complex Lie group and $F$ is a left-invariant strongly pseudoconvex complex Finsler metric on $G$, then $F$ must be a complex Berwald metric with vanishing holomorphic bisectional curvature; moreover, $F$ is a K\"ahler-Berwald metric if and only if $G$ is an abelian complex Lie group.
}




\maketitle

\section{Introduction and main results}
\hspace{1em}

The study of left-invariant metrics on Lie groups lies at the heart of differential geometry, offering a fertile interplay between algebraic structures and geometric properties. Milnor's seminal work \cite{JMilnor} systematically investigated the relationships between the Lie group structure and the existence of left-invariant Riemannian metrics with prescribed curvature conditions---ranging from nonnegative/nonpositive sectional curvature to Ricci or scalar curvature constraints, see also \cite{Heintze}. In a parallel direction, left-invariant Finsler metrics have attracted considerable attention, with notable contributions by Bao-Shen \cite{bao2002finsler}, Deng-Hou \cite{Deng-Hou1}, Deng \cite{Deng1}, Huang \cite{Huang1, Huang2}, Latifi-Toomanian \cite{Latifi-Toomanian},  Huang-Mo \cite{Huang-Mo}, Xu \cite{Xu1, Xu2}, and many others. These works have enriched our understanding of how Finslerian geometry interacts with Lie group symmetries.

The above-mentioned studies are primarily focused on real Riemann-Finsler geometry. As for complex Finsler geometry on homogeneous spaces, it was first shown in Deng-Hou \cite{Deng-Hou2} that a homogeneous complex Finsler space can be represented as a coset space, and a necessary and sufficient condition was established for such a coset space to admit a homogeneous complex Finsler metric. Moreover, it was proved that any symmetric complex Finsler space must be a Hermitian symmetric space, thereby providing a complete classification of symmetric complex Finsler spaces via the classical theory of Hermitian symmetric spaces. For a systematic study of homogeneous real Finsler manifolds and an introduction to homogeneous complex Finsler manifolds, we refer to Deng \cite{Deng1}. More recently, explicit examples of holomorphically invariant strongly pseudoconvex (and even strongly convex) K\"ahler-Berwald metrics were constructed on several important bounded symmetric domains, including the unit polydisk in $\mathbb{C}^n(n\geq 2)$ (cf. Zhong \cite{Zhong23}, Lin-Zhong \cite{Lin-Zhong}), the classical domains of type $I-IV$ (cf. Ge-Zhong \cite{Ge-Zhong}, Zhong \cite{Zhong25}), and, in the compact case, on complex Grassmann manifolds (cf. Cao-Ge-Zhong \cite{Cao-Ge-Zhong}). Furthermore, for the polydisk and the classical domains of type $I-IV$ admitting such more general holomorphically invariant strongly pseudoconvex K\"ahler-Berwald metrics, a complex Finslerian version of the Schwarz lemma was established for holomorphic mappings (cf. Lin-Zhong \cite{Lin-Zhong}, Zhong \cite{Zhong25}).

In this paper, we consider a connected $2n$-dimensional real Lie group $G$ which admits a left-invariant complex structure $J$ and a left-invariant strongly pseudoconvex complex Finsler metric $F:T^{1,0}G\rightarrow[0,+\infty)$ (not necessarily Hermitian). By left-invariance, $F$ is entirely determined by its value at the tangent space  of the identity  $e\in G$. Let $\mathfrak g$ denote the real Lie algebra of $G$, and let $\mathfrak g^{1,0}$ be the $\sqrt{-1}$-eigenspace of $J_e$ in the complexification $\mathfrak g^{\mathbb C}=\mathfrak g\otimes\mathbb C$.  Thus  we may regard $F$ as a complex Minkowski norm on $\mathfrak g^{1,0}$.

Our main purpose is to derive explicit curvature formulas for such $F$ entirely in terms of the complex Lie subalgebra $\mathfrak g^{1,0}$. To this end, we introduce a connection operator $\mathcal N_v:\mathfrak g^{1,0}\to\mathfrak g^{1,0}$ (see Definition~\ref{De:4_8}). Let $\sigma: \mathfrak{g}^{\mathbb{C}}\to\mathfrak{g}^{\mathbb{C}}$ be the conjugation map with respect to the real form $\mathfrak{g}$. The operator $\mathcal{N}_v$ plays an important role in characterizing when a left-invariant strongly pseudoconvex complex Finsler metric is K\"ahler-Finsler or weakly K\"ahler-Finsler. More precisely, we establish the following criterion (see also Theorem \ref{Th:4_11}).

\begin{theorem}\label{thm:Kahler}
Let $G$ be an even-dimensional connected real Lie group which is also a complex manifold. Let $\mathfrak{g}$ denote the Lie algebra of $G$, and $J$ a left-invariant complex structure on $G$. Let $F:T^{1,0}G\to[0,+\infty)$ be a left-invariant strongly pseudoconvex complex Finsler metric. Then
\begin{enumerate}
\item[(1)] $F$ is a K\"ahler-Finsler metric if and only if for any $u,w\in\mathfrak g^{1,0}$
\begin{align*}
&g_v(w, [u, \sigma(v)]^{1,0}) - g_v([v,w]^{1,0}, u) - g_v(v, [u, \sigma(w)]^{1,0}) \\
&\quad - \overline{\mathcal S}_v(u, [v,\sigma(w)]^{1,0}) - \mathcal C_v^+(w,u,\mathcal{N}_v(v)) + \mathcal C_v^-(w,u,[v,\sigma(v)]^{1,0}) = 0,
\end{align*}
where $v\in\mathfrak{g}^{1,0}\setminus\{0\}$ is the radial vector with variable direction.
\item[(2)] $F$ is a weakly K\"ahler-Finsler metric if and only if for any $w\in\mathfrak g^{1,0}$
$$
g_v(w, [v,\sigma(v)]^{1,0}) - g_v([v,w]^{1,0}, v) - g_v(v, [v,\sigma(w)]^{1,0}) - \mathcal S_v(\mathcal{N}_v(v), w) = 0,
$$
where $v\in\mathfrak{g}^{1,0}\setminus\{0\}$ is the radial vector with variable direction.
\end{enumerate}
\end{theorem}

In addition, we obtain the holomorphic bisectional curvature formula of $F$ expressed in terms of the structure constants of $\mathfrak g^{\mathbb C}$. Let $\mathcal D$ denote the flat connection on $\mathfrak g^{1,0}\setminus\{0\}$ given by the usual directional derivative (see \eqref{eq:4_17}). We obtain the following theorem.

\begin{theorem}\label{thm:curvature}
Under the same assumptions as in Theorem~\ref{thm:Kahler}, let $v\in\mathfrak g^{1,0}\setminus\{0\}$ be the radial vector with variable direction and $w\in\mathfrak g^{1,0}\setminus\{0\}$ an arbitrary vector. Then the holomorphic bisectional curvature in directions $v,w$ is
$$
B(v,w) = \frac{g_v(\mathcal R(w,\sigma(w))v, v)}{g_v(v,v)\,g_v(w,w)},
$$
where $\mathcal R(w,\sigma(w))v$ is given by
\begin{align*}
\mathcal{R}(w, \sigma(w))v=&\mathcal{D}_{\sigma(\mathcal{N}_v(w))}\bigl(\mathcal{N}_v(w)\bigr)-\mathcal{D}_{[v,\sigma(w)]^{1,0}}\bigl(\mathcal{N}_v(w)\bigr)
  -\mathcal{N}_v\bigl([w,\sigma(w)]^{1,0}\bigr)\\
&\quad   +[\mathcal{N}_v(w),\sigma(w)]^{1,0}-[v,[\sigma(w),w]^{0,1}]^{1,0}.
\end{align*}
\end{theorem}

When $G$ is a complex Lie group, the above formulas simplify dramatically, leading to a rigidity result.

\begin{theorem}\label{thm:rigidity}
Suppose $G$ is a complex Lie group and $F:T^{1,0}G\to[0,+\infty)$ is a left-invariant strongly pseudoconvex complex Finsler metric. Then $F$ must be a complex Berwald metric with vanishing holomorphic bisectional curvature. In particular, $F$ is a K\"ahler-Berwald metric if and only if $G$ is an abelian complex Lie group.
\end{theorem}

\begin{remark}
Theorem~\ref{thm:rigidity} has been independently obtained by X. Xu and  M. Xu in \cite{Xu2}. Our proof, however, is based on the moving frame method and the complex Rund connection, which essentially differs from the approach in \cite{Xu2}.
\end{remark}

These results furnish an algebraic framework for further investigations into the differential geometry of left-invariant strongly pseudoconvex complex Finsler metrics on Lie groups which are simultaneously complex manifolds.

The remainder of this paper is organized as follows. In section~\ref{Preliminaries}, we  recall some essential concepts from complex Finsler geometry. In Section~\ref{pullback bundle}, we develop the complex Rund connection on the pulled-back holomorphic tangent bundle. In section~\ref{LieGroup}, we apply this connection to left-invariant strongly pseudoconvex complex Finsler metrics,  thus obtaining the algebraic characterizations, curvature formulas, and the rigidity theorem for complex Lie groups.

\section{Preliminaries}\label{Preliminaries}
\hspace{1em}
In this section,  we recall some basic facts and definitions in complex differential geometry.
For Hermitian geometry, we refer to \cite{FoundationOfDG2};  for real and complex Finsler geometry, we refer to \cite{AP} and \cite{BSS}, respectively.

Let $M$ be a smooth manifold of real dimension $2n$ with real tangent bundle $TM$. An almost complex structure on $M$ is a smooth bundle endomorphism
$J: TM \to TM$ satisfying $J \circ J = -\text{Id}$. Using $J$,  we extend the scalar multiplication on $TM$ to the complex number field $\mathbb{C}$ by defining
$$
(a + b\sqrt{-1})X = aX + bJX
$$
for any $X \in TM$ and $a,  b \in \mathbb{R}$. This allows us to complexify $TM$, obtaining $T_{\mathbb{C}}M = TM \otimes \mathbb{C}$.
The map $J$ extends complex linearly to $T_{\mathbb{C}}M$ and admits two eigenbundles corresponding to the eigenvalues $+\sqrt{-1}$ and $-\sqrt{-1}$, respectively, leading to the decomposition
$$
T_{\mathbb{C}}M = T^{1, 0}M \oplus T^{0, 1}M.
$$

\begin{definition}
	The Nijenhuis torsion (or simply torsion) of an almost complex structure $J$ on $M$ is the tensor $N_J: TM \times TM \to TM$ defined by
$$
N_J(X,  Y) := [X,  Y] - [JX,  JY] + J[JX,  Y] + J[X,  JY]
$$
for any vector fields $X,  Y$ on $M$.
\end{definition}

The almost complex structure $J$ is said to be integrable if $N_J \equiv 0$.
By the Newlander-Nirenberg Theorem, this condition guarantees that $M$ admits a unique complex manifold structure such that $J$ corresponds to multiplication by $\sqrt{-1}$ on the holomorphic tangent bundle.

\begin{definition}
	Let $V$ be an $n$-dimensional complex vector space. A complex Minkowski norm on $V$ is a continuous function $F:V\rightarrow[0,+\infty)$ which is smooth on  $V\setminus \{0\}$ and satisfies
$$
F(\lambda v) = \vert\lambda\vert F(v)
$$
for all $v \in V$ and $\lambda \in \mathbb{C}$.
\end{definition}
The pair $(V,  F)$ is called a complex Minkowski space.

Given a basis $\{\alpha_i\}$ of $V$,  we write $v = v^i \alpha_i$ and $F(v) = F(v^1,  \ldots,  v^n)$. The Levi matrix $(g_{i\bar{j}})$ associated to $F$ is defined by
$$
g_{i\bar{j}}(v) = \frac{\partial^2 F^2}{\partial v^i \partial \bar{v}^j}.
$$
\begin{definition}
	A complex Finsler metric on a complex manifold $M$ of complex dimension $n$ is a continuous function $F: T^{1, 0}M \to[0,+\infty)$ such that its
	restriction $F_z(v):=F(z;v)$ to each fiber $T_z^{1, 0}M\cong\{z\}\times \mathbb{C}^n$ is a complex Minkowski norm. The pair $(M,  F)$ is called a complex Finsler manifold. In the special case
	where each $F_z$ arises from a Hermitian inner product,  $(M,  F)$ reduces to a Hermitian manifold.
\end{definition}
\begin{definition}
	A complex Finsler metric $F:T^{1,0}M\rightarrow[0,+\infty)$ on a complex manifold $M$ is called strongly pseudoconvex if its Levi matrix $$(g_{i\bar{j}}(z;  v))=\left(\frac{\partial^2 F^2}{\partial v^i \partial \bar{v}^j}\right)$$ is positive definite for all $z \in M$ and
	all nonzero $v \in T_z^{1, 0}M$.

\end{definition}
\textbf{Convention}: Unless stated otherwise,  all complex Finsler metrics considered  in the following will be assumed to be strongly pseudoconvex.

\section{The complex Rund connection}\label{pullback bundle}
\hspace{1em}
There are several complex Finsler connections associated to a given strongly pseudoconvex complex Finsler manifold $(M,F)$. For instance,  the Chern-Finsler connection \cite{AP}, the complex Rund connection in \cite{Aikou-b} and the complex Berwald connection and the complex Hashiguchi connection \cite{SZ}. However, to study left-invariant complex Finsler metrics on a real Lie group which is also a complex manifold, it is more convenient to work with the moving frame method via the smooth local frame.

In this section, we shall give a detailed treatment of the complex linear connection via a smooth local frame for the pulled-back holomorphic tangent bundle $\pi^\ast T^{1,0}M$. We derive the complex Rund connection via a local holomorphic frame for $\pi^\ast T^{1,0}M$. Whenever $(M,F)$ is a K\"ahler-Finsler manifold, we are able to obtain the complex Rund connection on $\pi^\ast T^{1,0}M$, which is the natural complex analogue of the well-known Chern connection in the real Finsler cases \cite{BSS}.

\subsection{Connection forms and structure equations}
\hspace{1em}
Let $M$ be a complex manifold of complex dimension $n$. Denote by $T^{1,0}M$ the holomorphic tangent bundle of $M$.
The slit holomorphic tangent bundle is $\tilde{M} := T^{1,0}M \setminus \{ \text{zero section} \}$.
Let $\pi: \tilde{M} \to M$ be the natural projection.

The pulled-back bundle $\pi^*T^{1,0}M$ is the complex vector bundle over $\tilde{M}$ whose fiber at $v \in \tilde{M}$ is
$$
\pi^*T^{1,0}M\vert_v := \{ ((z;v); w) \mid w \in T_z^{1,0}M \} \cong T_z^{1,0}M, \quad z = \pi(v).
$$

Let $\{e_i\}$ be a smooth local frame for $\pi^*T^{1,0}M$ and  $\{\varphi^i\}$ be a set of local horizontal $(1,0)$-forms on $\tilde{M}$ such that $\varphi^i(e_j) = \delta^i_j$.
For a complex Finsler metric $F: T^{1,0}M \to [0,\infty)$, define
$$
C_{i\bar{j}l} := \frac{\partial^3 F^2}{\partial v^i \partial \bar{v}^j \partial v^l}, \quad
C_{i\bar{j}\bar{l}} := \frac{\partial^3 F^2}{\partial v^i \partial \bar{v}^j \partial \bar{v}^l},
$$
which are called the complex Cartan tensors associated to $F$.
Then $\mathcal{G} = g_{i\bar{j}}\, \varphi^i \otimes \bar{\varphi}^j$ defines a Hermitian metric along the fibers of $\pi^\ast T^{1,0}M$. Setting
$$\mathcal{C}_v^+ = C_{i\bar{j}l}\, \varphi^i \otimes \bar{\varphi}^j \otimes \varphi^l, \quad \mathcal{C}_v^- = C_{i\bar{j}\bar{l}}\, \varphi^i \otimes \bar{\varphi}^j \otimes \bar{\varphi}^l.$$

\begin{theorem}\label{Th:3_1}
Let $(M,F)$ be a complex Finsler manifold of complex dimension $n$. Let $\{e_i\}$ be a smooth local frame for $\pi^*T^{1,0}M$ obtained by pulling back a smooth local frame for $T^{1,0}M$. Let $\{\varphi^i\}$ be the dual local horizontal $(1,0)$-forms on $\tilde M$ satisfying $\varphi^i(e_j)=\delta^i_j$. Then there exists a unique set of connection $1$-forms $\{\omega_j^i\}$ on $\tilde M$ satisfying
\begin{enumerate}
\item[(1)] first structure equation
\begin{equation}\label{eq:3_1}
d\varphi^i=\varphi^j\wedge\omega_j^i+\tau^i,
\end{equation}
where $\tau^i$ contain only horizontal $(2,0)$-components;
\item[(2)] almost metric compatibility
\begin{equation}\label{eq:3_2}
dg_{i\bar j}-g_{l\bar j}\omega^l_i-g_{i\bar l}\omega^{\bar l}_{\bar j}
=C_{i\bar j l}\delta v^l+C_{i\bar j\bar l}\delta\bar v^l,
\end{equation}
where $\delta v^l=dv^l+v^j\omega^l_j$.
\end{enumerate}
\end{theorem}

\begin{remark}
If $F$ comes from a Hermitian metric,  we have $ C_{i\bar{j}l}=C_{i \bar{j} \bar{l}}=0$, hence almost metric compatibility reduces to metric compatibility.
\end{remark}

\begin{proof}
Let $(z^i;v^i)$ be local holomorphic coordinates on $T^{1,0}M$. Choose a smooth local frame $\{e_i\}$ for $\pi^*T^{1,0}M$ obtained by pulling back a smooth local frame for $T^{1,0}M$, and let $\{\varphi^i\}$ be the dual local horizontal $(1,0)$-forms such that $\varphi^i(e_j)=\delta^i_j$.

A priori, we set
$$
\omega_j^i
=\Gamma^i_{\;jk}\varphi^k+\Gamma^i_{\;j\bar l}\bar\varphi^l
+\tilde\Gamma^i_{\;jk}dv^k+\tilde\Gamma^i_{\;j\bar l}d\bar v^l,
$$
where $\Gamma^i_{\;jk}$, $\Gamma^i_{\;j\bar l}$,
$\tilde\Gamma^i_{\;jk}$ and $\tilde\Gamma^i_{\;j\bar l}$ are to be
determined.

Since $\{e_i\}$ is pulled back from a smooth local frame for $T^{1,0}M$, we write
$$
\varphi^i=a^i_j(z)dz^j,\qquad e_i=b_i^j(z)\partial_j,
$$
where $\partial_j:=\pi^*(\partial/\partial z^j)$, $dz^i(\partial_j)=\delta^i_j$ and $a^i_jb^j_k=\delta^i_k$.

We first compare the coefficients in the first structure equation. A direct computation gives
\begin{equation}\label{eq:3_3}
\begin{aligned}
d\varphi^i=\frac{\partial a^i_p}{\partial z^q}dz^q\wedge dz^p+\frac{\partial a^i_p}{\partial\bar z^q}d\bar z^q\wedge dz^p.
\end{aligned}
\end{equation}
On the other hand, setting $\tau^i=T^i_{\;jk}\varphi^j\wedge\varphi^k$, we have
\begin{equation*}
\begin{aligned}
\varphi^j\wedge\omega_j^i+\tau^i={}&
a^j_p\Gamma^i_{\;jk}a^k_qdz^p\wedge dz^q
+a^j_p\Gamma^i_{\;j\bar k}a^{\bar k}_{\bar q}dz^p\wedge d\bar z^q\\
&+a^j_p\tilde\Gamma^i_{\;jk}dz^p\wedge dv^k
+a^j_p\tilde\Gamma^i_{\;j\bar k}dz^p\wedge d\bar v^k
+T^i_{\;jk}\varphi^j\wedge\varphi^k.
\end{aligned}
\end{equation*}
Comparing the above equation with \eqref{eq:3_3}, we obtain
\begin{equation}\label{eq:3_4}
\Gamma^i_{\;j\bar l}
=-\frac{\partial a^i_q}{\partial\bar z^p}b^q_jb^{\bar p}_{\bar l},
\qquad
\tilde\Gamma^i_{\;jk}=0,
\qquad
\tilde\Gamma^i_{\;j\bar l}=0.
\end{equation}
Substituting \eqref{eq:3_4} into $\omega_j^i$ gives
\begin{equation}\label{eq:3_5}
\omega^i_j
=\Gamma^i_{\;jk}\varphi^k
-b^m_j\bar b^k_l\frac{\partial a^i_m}{\partial\bar z^k}\bar\varphi^l.
\end{equation}
In particular, $\omega_j^i$ contains no vertical components.

Substituting \eqref{eq:3_5} into \eqref{eq:3_2} and comparing the coefficients of $dz^k$ on both sides yields
\begin{equation}\label{eq:3_6}
\Gamma^i_{\;jk}
=g^{\bar qi}b^s_k\left(
\frac{\partial g_{j\bar q}}{\partial z^s}
+g_{j\bar l}\bar b^m_q\frac{\partial\bar a^l_m}{\partial z^s}
-C_{j\bar q l}N^l_{\;m} a^m_s
+C_{j\bar q\bar l}\bar v^p\bar b^m_p\frac{\partial\bar a^l_m}{\partial z^s}
\right),
\end{equation}
where $N^l_{\;m}:=\Gamma^l_{\;im}v^i$ is determined by
$$
N^l_{\;m}
=g^{\bar jl}b^k_m\left(
\frac{\partial g_{i\bar j}}{\partial z^k}
+g_{i\bar l}\bar b^s_j\frac{\partial\bar a^l_s}{\partial z^k}
+C_{i\bar j\bar l}\bar v^p\bar b^s_p\frac{\partial\bar a^l_s}{\partial z^k}
\right)v^i.
$$

Consequently, all connection coefficients $\Gamma^i_{\;jk}$ and $\Gamma^i_{\;j\bar l}$ are uniquely determined by $g_{i\bar{j}}$, $C_{i\bar{j}l}$, $C_{i\bar{j}\bar{l}}$  and the choice of the pulled-back smooth local frame $\{e_i\}$, while the vertical coefficients $\tilde\Gamma^i_{\;jk}$ and $\tilde\Gamma^i_{\;j\bar l}$ vanish identically. Hence the horizontal connection $1$-forms $\omega^i_j$ exist and are unique.
\end{proof}

\begin{corollary}[Complex Rund connection]\label{Co:3_1}
Let $(M,F)$ be a strongly pseudoconvex complex Finsler manifold. The pulled-back bundle $\pi^{\ast}T^{1,0}M$ admits a unique complex linear connection, known as the complex Rund connection. More precisely, for any holomorphic local frame $\{e_i=\pi^\ast \varepsilon_i\}$ for $\pi^*T^{1,0}M$ with dual frame $\{\varphi^k=\pi^\ast\phi^k\}$, where $\{\varepsilon_i\}$ is a holomorphic local frame for $T^{1,0}M$ and $\{\phi ^k\}$ its dual frame, there exists a unique set of horizontal $(1,0)$-forms $\{\omega_j^i\}$ on $\tilde{M}$ satisfying
\begin{enumerate}
\item[(1)]  first structure equation
\begin{equation*}
	d\varphi^i = \varphi^j \wedge \omega_j^i + \tau^i,
\end{equation*}
where  $\tau^i$ contain only horizontal forms of type $(2,0)$.
\item[(2)]  almost metric compatibility
\begin{equation*}
	dg_{i\bar{j}}-g_{l\bar{j}} \omega^l_i-g_{i\bar{l}} \omega^{\bar{l}}_{\bar{j}}=C_{i\bar{j}l} \delta v^l  + C_{i \bar{j} \bar{l}} \delta \bar{v}^l ,
\end{equation*}
where $\delta v^l = dv^l + v^j \omega^l_j$.
\end{enumerate}
\end{corollary}

\begin{proof}
By the proof of Theorem \ref{Th:3_1}, if  $\{\varepsilon_i=b_i^j(z) \frac{\partial}{\partial z^j}\}$ is a holomorphic local frame for $T^{1,0}M$ and $\{\phi^i=a^i_j(z) dz^j\}$ is the dual frame of $\{\varepsilon_i\}$, then $\{e_i:=\pi^\ast \varepsilon_i\}$ is a holomorphic local frame for $\pi^\ast T^{1,0}M$ and $\{\varphi^k:=\pi^\ast\phi^k\}$ is the dual frame of $\{e_i\}$ such that $a_j^ib_i^k=\delta_j^k$, and $\Gamma_{\;j\bar{l}}^i=0$. Thus
\begin{equation}
\Gamma^i_{\;jk} = g^{\bar{q}i}b_k^s \left(\frac{\partial g_{j\bar{q}}}{\partial z^s} -a^m_sC_{j\bar{q}p} N^p_{\;m}\right)=g^{\bar{q}i}\left[e_k(g_{j\bar{q}})-N_{\;k}^p\frac{\partial g_{j\bar{q}}}{\partial v^p}\right],\label{eq:3_7}
\end{equation}
so that $\omega_j^i=\Gamma_{\;jk}^i\varphi^k$ are horizontal $1$-forms of type $(1,0)$.
\end{proof}

\begin{corollary}
Let $(M,F)$ be a K\"ahler-Finsler manifold. The pulled-back bundle $\pi^*T^{1,0}M$ admits a unique complex linear connection, called the complex Rund connection. Its connection forms are
characterized by the structure equations:
\begin{enumerate}
\item[(1)]  torsion freeness
\begin{equation*}
	d(dz^i)-dz^j \wedge \omega_j^i=-dz^j\wedge \omega_j^i=0.
\end{equation*}
\item[(2)]  almost metric compatibility
\begin{equation*}
	dg_{i\bar{j}}-g_{l\bar{j}} \omega^l_i- g_{i\bar{l}} \omega^{\bar{l}}_{\bar{j}}=C_{i\bar{j}l}\delta v^l + C_{i \bar{j} \bar{l}} \delta\bar{v}^l,
\end{equation*}
where $\delta v^l = dv^l + v^j \omega^l_j$.
\end{enumerate}
\end{corollary}
\begin{proof}
Taking $b^i_j=\delta^i_j$  in the proof of Corollary \ref{Co:3_1}, then $a_i^k=\delta_i^k$, consequently
$$
\tau^i=-\Gamma^{i}_{\;jk} dz^j \wedge dz^k.
$$
Thus $\tau^i=0$ if and only if
$$
\Gamma^i_{\;jk} = \Gamma^i_{\;kj},
$$
namely $F$ is a K\"ahler-Finsler metric.
\end{proof}

\begin{definition}\label{De:3_3}
Let $F:T^{1,0}M\to[0,+\infty)$ be a complex Finsler metric on a complex manifold $M$, and let $\{e_i\}$ be a smooth local frame for $\pi^*T^{1,0}M$ obtained by pulling back a smooth local frame for $T^{1,0}M$. Let $\Gamma_{\;jk}^i$ be  given by \eqref{eq:3_6}. We call $F$ a complex Berwald metric if
$$
\frac{\partial}{\partial v^l}\Gamma^i_{\;jk}=\frac{\partial}{\partial \bar v^l}\Gamma^i_{\;jk}=0,\quad \forall i,j,k,l=1,\cdots,n.
$$
\end{definition}

\begin{proposition}
Definition \ref{De:3_3} is independent of the choice of the pulled-back smooth local frame $\{e_i\}$ for $\pi^*T^{1,0}M$.
\end{proposition}

\begin{proof}
Let $(z^i;v^i)$ be local holomorphic coordinates on $\tilde M$, where $z^i$ are local holomorphic coordinates on $M$ and $v^i$ are the fiber coordinates. Let $\{e_i\}$ and $\{\tilde e_i\}$ be two pulled-back smooth local frames for $\pi^*T^{1,0}M$ over a common open set of $\tilde{M}$, with dual frames $\{\varphi^i\}$ and $\{\tilde\varphi^i\}$, respectively.

Since both frames are pulled back from smooth local frames for $T^{1,0}M$, their coefficients depend only on the base variable $z$. Thus we may write
$$
e_i=b_i^j(z)\partial_j,
\qquad
\varphi^i=a^i_j(z)dz^j,
$$
and
$$
\tilde e_i=\tilde b_i^j(z)\partial_j,
\qquad
\tilde\varphi^i=\tilde a^i_j(z)dz^j,
$$
where $(a^i_j)$ and $(\tilde a^i_j)$ are the inverse matrices of $(b_i^j)$ and $(\tilde b_i^j)$, respectively. Consequently, the change-of-frame matrices
$$
\tilde e_i=A^j_i(z)e_j,
\qquad
\tilde\varphi^i=B^i_j(z)\varphi^j
$$
depend only on $z$ as well. Here $A=(A^j_i)$ and $B=(B^i_j)$ are nondegenerate inverse matrices satisfying $A^j_iB^i_k=\delta^j_k$.

Let $v=v^ie_i=\tilde v^i\tilde e_i$. Since $\tilde e_i=A^j_i(z)e_j$, we have
$v^j=A^j_i(z)\tilde v^i$. Then the Jacobi matrices are independent of the fiber variable, $\frac{\partial v^j}{\partial\tilde v^i}=A^j_i$, and therefore
$$
\frac{\partial}{\partial\tilde v^l}
=A^j_l\frac{\partial}{\partial v^j}.
$$

Denote the connection $1$-forms of the complex Rund connection with respect to $\{\varphi^i\}$ and $\{\tilde\varphi^i\}$ by
$$
\omega^i_j=\Gamma^i_{\;jk}\varphi^k+\Gamma^i_{\;j\bar k}\bar\varphi^k,
\qquad
\tilde\omega^i_j
=\tilde\Gamma^i_{\;jk}\tilde\varphi^k
+\tilde\Gamma^i_{\;j\bar k}\bar{\tilde\varphi}^k,
$$
respectively.
The transformation law of the connection $1$-forms gives
$$
\tilde\omega^i_j
=B^i_kdA^k_j+B^i_kA^l_j\omega^k_l.
$$
Substituting the expressions for $\omega^i_j$ and $\tilde\omega^i_j$ into the above equation and using $\tilde\varphi^i=B^i_k\varphi^k$, we compare the coefficients of
$\varphi^k$ and obtain
\begin{equation}\label{eq:3_8}
\tilde\Gamma^i_{\;jk}
=A^s_k\left(
B^i_pA^l_j\Gamma^p_{\;ls}
+B^i_pb^l_s\partial_lA^p_j
\right).
\end{equation}

Since $A^s_k$, $B^i_p$ and $b^l_s$ depend only on $z$, the right-hand side of \eqref{eq:3_8} is independent of $v$ and $\bar v$ if and only if $\Gamma^p_{\;ls}$ is independent of $v$ and $\bar v$. Consequently, using $\frac{\partial}{\partial\tilde v^q}=A^j_q\frac{\partial}{\partial v^j}$, we obtain
$$
\frac{\partial}{\partial v^q}\Gamma^p_{\;ls}
=\frac{\partial}{\partial\bar v^q}\Gamma^p_{\;ls}=0
\iff
\frac{\partial}{\partial \tilde{v}^q}\tilde\Gamma^i_{\;jk}
=\frac{\partial}{\partial\bar{\tilde{v}}^q}\tilde\Gamma^i_{\;jk}=0.
$$
Thus the definition of a complex Berwald metric is independent of the choice of the pulled-back smooth local frame $\{e_i\}$ for $\pi^*T^{1,0}M$.
\end{proof}

\begin{definition}\cite{AP}\label{De:3_5}
Write the horizontal $(2,0)$-part of the torsion forms $\tau^j$ in \eqref{eq:3_1} as $\tau^j = T^j_{\;ik}\varphi^i\wedge\varphi^k$. A complex Finsler metric $F$ on $M$ is called a K\"ahler-Finsler metric if
\begin{equation*}
	T^j_{\; ik}=0, \quad \forall i,j,k=1,\cdots,n;
\end{equation*}
is called a weakly K\"ahler-Finsler metric	if
\begin{equation*}
	g_{j\bar l}T^j_{\;ik}  v^i\bar  v^l=0, \quad \forall k=1,\cdots,n.
\end{equation*}
\end{definition}
Under a change of pulled-back smooth local frame $\tilde{e}_i = A_i^j e_j$, the torsion transforms as $\tilde{T}^j_{\;ik} = (A^{-1})^j_p A^q_i A^r_k T^p_{\;qr}$. Hence the definition   $T^j_{\;ik}=0$ for  a K\"ahler-Finsler metric  is independent of the choice of  a pulled-back smooth local frame for $\pi^\ast T^{1,0}M$. This is similar for the definition of a weakly K\"ahler-Finsler metric.

Whenever $T_{\;ik}^j:=\Gamma_{\;ik}^j-\Gamma_{\;ki}^j$ with $\Gamma_{\;ik}^j$ given by \eqref{eq:3_7} and $b^i_j=a^i_j=\delta^i_j$, Chen and Shen \cite{ChenAndShen} proved that $T^j_{\; ik}  v^i=0$ actually implies $T_{\;ik}^j=0$.
If $F$ comes from a Hermitian metric,   then the definition of both K\"ahler-Finsler metric and weakly K\"ahler-Finsler metric reduces to the usual  definition of K\"ahler metric in Hermitian geometry.

\subsection{Complex vector subbundle $\mathcal{H}_{\mathbb{C}}$}
\hspace{1em}
Throughout this subsection, all frames for $\pi^*T^{1,0}M$ are pulled-back from smooth local frames for $T^{1,0}M$, so that their coefficients depend only on $z\in M$.

We now study the decomposition of the complexified tangent bundle $T_{\mathbb{C}}\tilde{M}$ induced by the complex Rund connection with respect to a smooth local frame. Our goal is to identify the complex vector subbundle $\mathcal{H}_{\mathbb{C}} \subset T_{\mathbb{C}}\tilde{M}$ with the lifted vectors from $T_{\mathbb{C}}M$, and to understand its transformation behavior under a smooth frame change.

From now on, we denote by $\{e_i^\ast\}$ a pulled-back smooth local frame for $\pi^*T^{1,0}M$, and by $\{e_i\}$ a smooth local frame for $T^{1,0}M$.

By the convention fixed at the beginning of this subsection, we choose a pulled-back smooth local frame $\{e_i^*\}$ for $\pi^*T^{1,0}M$ over $U\subset\tilde M$, so that its coefficients depend only on $z$. Namely, on a coordinate chart $(U, (z^i;v^i))$, we write $e_i^*=b_i^j(z)\partial_{z^j}$, where $\partial_{z^j}=\pi^\ast(\frac{\partial}{\partial z^j})$. The dual $(1,0)$-forms are $\varphi^i = a^i_j(z)dz^j$ with $a^i_j b_i^k = \delta_j^k$.

On another coordinate chart $(V, (\tilde z^i;\tilde v^i))$ with $U\cap V\neq\emptyset$, denote the corresponding objects by $\tilde e^\ast_i$, $\tilde\varphi^i$, etc. From $\tilde e^\ast_i = \tilde b_i^j(\tilde z)\partial_{\tilde z^j}$ and $\partial_{\tilde z^j} = \frac{\partial z^k}{\partial\tilde z^j}\partial_{z^k}$, we obtain
$\tilde e^\ast_i = \tilde b_i^j \frac{\partial z^k}{\partial\tilde z^j} a^p_k e^\ast_p$.
Similarly, the fiber coordinates satisfy $v = v^i e^\ast_i = \tilde v^i \tilde e^\ast_i$, which gives $\tilde v^j = v^p b_p^k \tilde a^j_s \frac{\partial\tilde z^s}{\partial z^k}$.

For simplicity, define $B_p^i := \tilde a^i_j \frac{\partial\tilde z^j}{\partial z^k} b^k_p$, $D_i^p := \tilde b^s_i \frac{\partial z^l}{\partial\tilde z^s} a^p_l$, which satisfy $B_p^i D_j^p = \delta^i_j$.

Let $\nabla e^\ast_p = \omega^i_p e^\ast_i$ be the complex Rund connection with respect to $\{e^\ast_i\}$. It follows from $\nabla \tilde e^\ast_p = \tilde\omega^i_p \tilde e^\ast_i$ and $\tilde e^\ast_p = D_p^i e^\ast_i$ that
$$
\tilde\omega^p_j = B_q^p d(D_j^q) + B_q^p D_j^l \omega^q_l.
$$
Substituting the expressions $\tilde\omega^p_j = \tilde\Gamma^p_{\;jq}\tilde\varphi^q + \tilde\Gamma^p_{\;j\bar q}\bar{\tilde\varphi}^q$ and $\omega^q_l = \Gamma^q_{\;ls}\varphi^s + \Gamma^q_{\;l\bar s}\bar\varphi^s$ into the above equation yields
\begin{align*}
\tilde\Gamma^p_{\;jq} &=B^p_s D^m_q  e_m(D^s_j) + D^m_q B^p_s D^l_j \Gamma^s_{\;lm},\\
\tilde\Gamma^p_{\;j\bar{q}} &= B^p_s\bar{D}^m_q  \bar e_m(D^s_j) + \bar{D}^m_q B^p_s D^l_j \Gamma^s_{\;l\bar{m}}.
\end{align*}

Let $\{\varphi^i,\bar\varphi^i,dv^i,d\bar v^i\}$ and $\{\tilde\varphi^i,\bar{\tilde\varphi}^i,d\tilde v^i,d\bar{\tilde v}^i\}$ be the natural local bases of $T_{\mathbb{C}}^*\tilde{M}$ on $U\subset \tilde{M}$ and $V\subset\tilde{M}$ with $U\cap V\neq\emptyset$, respectively.
Using the duality between $\{\tilde\varphi^i,\bar{\tilde\varphi}^i,d\tilde v^i,d\bar{\tilde v}^i\}$ and $\{\tilde e_j,\bar{\tilde e}_j,\partial_{\tilde v^j}:=\frac{\partial}{\partial \tilde{v}^j},\partial_{\bar{\tilde v}^j}:=\frac{\partial}{\partial \bar{\tilde{v}}^j}\}$ gives
\begin{align*}
\tilde\varphi^i &= B_p^i \varphi^p,\\
d\tilde v^i &= B_q^i dv^q + v^p e_q(B_p^i)\varphi^q + v^p \bar e_q(B_p^i)\bar\varphi^q,\\
\intertext{and}
\tilde e_j &= D_j^q e_q - v^q D_j^me_m(B_q^p)  D_p^s \partial_{v^s} -\bar{v}^qD^m_je_m(\bar{B}^p_q)\bar{D}^s_p\partial_{\bar v^s},\\
\partial_{\tilde v^j} &= D_j^p \partial_{v^p}.
\end{align*}

Let $J$ be the complex structure on $M$, extended to $\tilde{M}$ by requiring that
$$
J(e_i)=\sqrt{-1}\,e_i,J(\partial_{v^p})=\sqrt{-1}\,\partial_{v^p},
J(\bar{e}_i)=-\sqrt{-1}\,\bar{e}_i, J(\partial_{\bar v^p})=-\sqrt{-1}\,\partial_{\bar v^p}.
$$
Define
$$
\delta_i := e_i - \Gamma^p_{\;ji} v^j \partial_{v^p} - \overline{\Gamma^p_{\;j\bar{i}} v^j} \partial_{\bar{v}^p},
$$
which satisfies $\tilde\delta_i = D_i^m \delta_m$.
However, in a general smooth frame, $J(\delta_i) \neq \sqrt{-1}\delta_i$; this identity holds only when the frame is holomorphic. Nevertheless, $\mathcal{H}_{\mathbb{C}} :=\operatorname{span}\{\delta_i,\bar\delta_i\}$ spans a well-defined complex vector subbundle  of $T_{\mathbb{C}}\tilde{M}$.

Now set $\mathcal{H}_{\mathbb{C}}^* := \operatorname{span}\{\varphi^i,\bar\varphi^i\}$ and $\mathcal{V}_{\mathbb{C}}^* := \operatorname{span}\{\delta v^i,\delta \bar{v}^i\}$. Then
$\tilde\varphi^i = B_p^i \varphi^p$ and $\tilde{\delta} v^{i} = B_p^i\delta v^p$, so $\mathcal{H}_{\mathbb{C}}^*$ and $\mathcal{V}_{\mathbb{C}}^*$ are well-defined complex vector subbundles of $T_{\mathbb{C}}^*\tilde{M}$ with the decomposition
$$
T_{\mathbb{C}}^*\tilde{M} = \mathcal{H}_{\mathbb{C}}^* \oplus \mathcal{V}_{\mathbb{C}}^*.
$$
Let $\mathcal{H}_{\mathbb{C}} = \operatorname{span}\{\delta_i,\bar\delta_i\}$ and $\mathcal{V}_{\mathbb{C}} := \operatorname{span}\{\partial_{v^i},\partial_{\bar v^i}\}$. By construction, $\mathcal{H}_{\mathbb{C}}$ and $\mathcal{V}_{\mathbb{C}}$ are dual to $\mathcal{H}_{\mathbb{C}}^*$ and $\mathcal{V}_{\mathbb{C}}^\ast$ respectively. Hence we obtain the dual decomposition
$$
T_{\mathbb{C}}\tilde{M} = \mathcal{H}_{\mathbb{C}} \oplus \mathcal{V}_{\mathbb{C}}.
$$
Let $\{e_i\}$ be a smooth local frame for $T^{1,0}M$. For any $v \in \tilde{M}$ and any $w = w^i e_i\vert_{\pi(v)} \in T^{1,0}_{\pi(v)}M$ and $\bar u = \bar u^i \bar e_i\vert_{\pi(v)} \in T^{0,1}_{\pi(v)}M$, define the  lift map
$$
\chi_v : T^{1,0}_{\pi(v)}M \oplus T^{0,1}_{\pi(v)}M \longrightarrow \mathcal{H}_{\mathbb{C}}\vert_v
$$
by
$$
\chi_v(w + \bar u) := w^i \delta_i\vert_v + \bar u^i \bar\delta_i\vert_v.
$$
Since this map is fiberwise complex linear and $\{\delta_i,\bar\delta_i\}$ is a smooth local frame for $\mathcal{H}_{\mathbb{C}}$, $\chi_v$ defines a smooth complex vector bundle isomorphism
$$
\pi^*T^{1,0}M \oplus \overline{\pi^*T^{1,0}M} \cong \mathcal{H}_{\mathbb{C}}.
$$
In other words, the lift gives a smooth isomorphism between the pull-back of the complexified tangent bundle of $M$ and the complex vector subbundle $\mathcal{H}_{\mathbb{C}} \subset T_{\mathbb{C}}\tilde{M}$.

\begin{remark}
When the smooth local frame $\{e_i\}$ is holomorphic, the complex Rund connection $1$-forms $\omega^i_j$ are of type $(1,0)$. In this case, $\delta v^i=dv^i+v^j\omega^i_j$ are of type $(1,0)$, the lift $\chi_v$ reduces to the standard horizontal lift in \cite{AP}.
\end{remark}

\subsection{Holomorphic bisectional/sectional curvature}
\hspace{1em}
Let $\nabla$ be the complex Rund connection on the pulled-back bundle $\pi^*T^{1,0}M$. Its curvature is defined as the operator
$$
\Omega(X,Y)\xi  := \nabla_X \nabla_Y \xi  - \nabla_Y \nabla_X \xi - \nabla_{[X,Y]} \xi,
$$
for $X,Y \in T_{\mathbb{C}}\tilde{M}$ and $\xi \in \Gamma(\pi^*T^{1,0}M)$, where $\Gamma(\pi^*T^{1,0}M)$ is the set of smooth sections.

Choose a pulled-back smooth local frame $\{e^\ast_i\}$ for $\pi^*T^{1,0}M$ with the connection $1$-forms $\omega^i_j$ satisfying $\nabla e^\ast_j = \omega^i_j e^\ast_i$. A direct computation yields
$$
\Omega(X,Y)e^\ast_j = \bigl(d\omega^i_j(X,Y) - \omega^k_j \wedge \omega^i_k(X,Y)\bigr) e^\ast_i.
$$
Thus the curvature $2$-forms $\Omega^i_j$, defined by $\Omega(X,Y)e^\ast_j = \Omega^i_j(X,Y)e^\ast_i$, are given by the following Cartan structure equation
\begin{equation}\label{eq:3_9}
\Omega^i_j = d\omega^i_j - \omega^k_j \wedge \omega^i_k.
\end{equation}

Expanding \eqref{eq:3_9} shows
\begin{equation}\label{eq:3_10}
\Omega^i_j = S^i_{\;jkl}\,\varphi^k\wedge\varphi^l + R^i_{\;jk\bar l}\,\varphi^k\wedge\bar\varphi^l + P^i_{\;jkl}\,\varphi^k\wedge\delta v^l + Q^i_{\;jk\bar l}\,\varphi^k\wedge\delta\bar{v}^l,
\end{equation}
where the coefficients $S^i_{\;jkl}$, $R^i_{\;jk\bar{l}}$, $P^i_{\;jkl}$ and $Q^i_{\;jk\bar{l}}$ are smooth functions on $\tilde{M}$.

\begin{remark}
 If $F$ comes from a Hermitian metric, then $C_{i\bar{j}l}=C_{i\bar{j}\bar{l}}\equiv 0$. Consequently, the curvature $2$-forms $\Omega^i_j$ reduce to the type $(1,1)$-forms, which are the pull-back of Chern-curvature $2$-forms of the corresponding Hermitian metric.
\end{remark}

Let $\{e_i\}$ be a smooth local frame for $T^{1,0}M$ and $v\in\tilde{M}$. Write $w = w^i e_i\vert_{\pi(v)}$. The natural lift of $w$ to $\pi^*T^{1,0}M$ at $v$ is given by
$$
w^* := \pi^*(w)\vert_v = w^i e^\ast_i\vert_v.
$$

The primary component coefficients $R^i_{\;jk\bar{l}}$  determine the holomorphic sectional and bisectional curvatures of $F$.

Let $X,Y \in \mathcal{H}_{\mathbb{C}}$ be nonzero vectors, and let $\xi=\xi^le^\ast_l, \eta =\eta^se^\ast_s$ be local sections of $\pi^*T^{1,0}M$. Write $X = X^k \delta _k+\tilde{X}^{\bar{k}}\bar{\delta}_k$ and  $Y = Y^j \delta_j+\tilde{Y}^{\bar{k}}\bar{\delta}_k$. The curvature operator with respect to the pulled-back smooth local frame $\{e^\ast_i\}$ for $\pi^*T^{1,0}M$ is defined by
$$
R(X,\overline{Y})\xi := \Omega^i_j(X,\overline{Y}) \varphi^j(\xi) \otimes e^\ast_i=\xi^j X^k \overline{Y}^l R^i_{\;jk\bar{l}}\, e^\ast_i,
$$
and the corresponding curvature tensor is defined by
\begin{equation}\label{eq:3_11}
R(\xi,\overline{\eta},X,\overline{Y}) := \mathcal{G}\bigl(R(X,\overline{Y})\xi,\overline{\eta}\bigr),	
\end{equation}
where $R^i_{\;jk\bar{l}}$ are given by \eqref{eq:3_10}.

\begin{definition}
Let $(M,F)$ be a complex Finsler manifold and $v \in \tilde{M}$. For $w \in T^{1,0}_{\pi(v)}M$, denote by $ \chi_v(w)$ the lift to $\mathcal{H}_{\mathbb{C}}$, and by $ w^* \in \pi^*T^{1,0}M$ the natural lift. The holomorphic bisectional curvature of $(M,F)$ in the directions $v$ and $w$ is defined as
$$
B(v,w) := \frac{R(v^*, \bar{v}^*, \chi_v(w), \chi_v(\bar{w}))}{\mathcal{G}(v^*, v^*) \mathcal{G}(w^*, w^*)},
$$
where $R(\cdot,\cdot,\cdot,\cdot)$ is the curvature tensor defined in \eqref{eq:3_11}.

The holomorphic sectional curvature of $(M,F)$ in the direction $v$ is defined as
$$
K(v) := 2\frac{R(v^*, \bar{v}^*, \chi_v(v), \chi_v(\bar{v}))}{\mathcal{G}(v^*, v^*)^2}.
$$
\end{definition}

\begin{remark}
The above definition is based on the complex Rund connection with respect to a pulled-back smooth local frame. For the corresponding definition which is based on the Chern-Finsler connection with respect to a holomorphic local frame---the complex vertical subbundle $\mathcal{V}^{1,0}$ (which is isomorphic to $\pi^\ast T^{1,0}M$),  we refer to \cite{AP}.
\end{remark}

Define the contracted curvature forms
\begin{equation*}
\Omega^i := d(\delta v^i) - \delta v^j \wedge \omega^i_j,
\end{equation*}
where $\delta v^i=dv^i + v^k\omega^i_k$. A direct computation using the Cartan structure equation yields
\begin{equation}\label{eq:3_12}
\Omega^i = v^k \Omega^i_k,
\end{equation}
and similarly $\overline{\Omega^i} = \bar v^k \overline{\Omega^{i}_{k}}$.
Consequently, $\Omega^i$ can be expanded  as
\begin{equation*}
\Omega^i = S^i_{\;kl}\,\varphi^k\wedge\varphi^l + R^i_{\;k\bar l}\,\varphi^k\wedge\bar\varphi^l + P^i_{\;kl}\,\varphi^k\wedge\delta v^l + Q^i_{\;k\bar l}\,\varphi^k\wedge\delta \bar{v}^l,
\end{equation*}
where the coefficients are given by contractions of the curvature components
$$
S^i_{\;kl} = S^i_{\;jkl}v^j,\quad
R^i_{\;k\bar l} = R^i_{\;jk\bar l}v^j,\quad
P^i_{\;kl} = P^i_{\;jkl}v^j,\quad
Q^i_{\;k\bar l} = Q^i_{\;jk\bar l}v^j.
$$

\section{Left-invariant metrics on Lie groups}\label{LieGroup}
\hspace{1em}
In the following, we always assume that $G$ is an even-dimensional real Lie group which admits a left-invariant complex structure $J$, namely $G$ is simultaneously a complex manifold.
Our purpose is to study left-invariant complex Finsler metrics on $G$.
We derive the complex Rund connection and its  curvature of left-invariant complex Finsler metrics on $G$. In particular, if $G$ is a complex Lie group, we obtain some rigidity results of left-invariant complex Finsler metrics on $G$.

\subsection{Lie groups which are also complex manifolds}\label{LLL}
\hspace{1em}
Let $ G $ be a real Lie group of even dimension. Denote by $L_g: G \to G$, $L_g(h)=gh$, the left translation by $g\in G$ which is  real analytic. A vector field $ Z $ on $ G $ is called left-invariant if $ dL_h Z = Z $ for all $ h \in G $. Given a tangent vector $ X \in T_eG $, there exists exactly one left invariant vector field $ \iota(X) $ on $ G $ such that $ \iota(X)_e = X $, and this $ \iota(X) $ is real analytic.
Explicitly, $\iota(X)$ can be constructed by
$$
[\iota(X)f](g) := \frac{d}{dt}\Big\vert_{t=0}f (g \gamma(t)),
$$
where $f\in C^{\infty}(G)$, $g\in G$ and $\gamma(t)$ is any curve in $G$ with the tangent vector $X$ at $e=\gamma(0)\in G$.
If $ X, Y \in T_eG $, then the vector field $ [\iota(X), \iota(Y)] $ is left invariant. The tangent vector $ [\iota(X), \iota(Y)]_e $ is simply denoted by $ [X, Y] $ if it causes no confusion. The vector space $ T_eG $, with the rule of composition $ (X, Y) \mapsto [X, Y] $, is denoted by $ \mathfrak{g} $ and called the Lie algebra of $G$. Immediately, we have $\iota([X,Y])=[\iota(X),\iota(Y)]$.
\begin{definition}\label{def-4.1}
Let $G$ be a real Lie group of even dimension. An almost complex structure $J$ on $G$ (i.e. a smooth $(1,1)$-tensor field satisfying $J^2 = -\mathrm{Id}$) is called left-invariant if it is preserved by every left translation, namely for all $g,h\in G$,
\begin{equation}
(dL_g)_h \circ J_h = J_{gh} \circ (dL_g)_h,\label{eq-prop-a}
\end{equation}
or equivalently, for every $X\in T_hG$,
$$
(dL_g)_h\bigl(J_h(X)\bigr) = J_{gh}\bigl((dL_g)_h(X)\bigr).
$$
In pulled-back notation, this is simply
$$
L_g^* J = J \quad \text{for all } g\in G.
$$
\end{definition}

Since $J$ is left-invariant, it is completely determined by its value at the identity $e\in G$. Let
$$
I := J_e \in \mathrm{End}(\mathfrak g), \qquad \mathfrak g = T_eG,
$$
then $I^2 = -\mathrm{Id}_{\mathfrak g}$, and $I$ is called a Koszul operator \cite{koszul1959exposes}. For any $g\in G$, we have the explicit formula
$$
J_g = (dL_g)_e \circ I \circ (dL_g)_e^{-1}.
$$
In other words, if $v = (dL_g)_e(X)$ for some $X\in\mathfrak g$, then
$$
J_g(v) = (dL_g)_e\bigl(I(X)\bigr).
$$

If furthermore $J$ is integrable (so that $G$ is a complex manifold) and left-invariant, then it is called a left-invariant complex structure on $G$. In this case, the linear map $I$ on $\mathfrak g$ must also satisfy the integrability condition (vanishing of the Nijenhuis tensor), which is
$$
N_I(X,Y):=[X,Y]-[IX,IY] + I[IX,Y] + I[X,IY] = 0, \qquad \forall X,Y\in\mathfrak g.
$$

\begin{remark}\label{remark-4.2}
It follows from Definition \ref{def-4.1} that on a real Lie group $G$ admitting an integral left-invariant complex structure $J$, every left translation $L_g$ is holomorphic.
\end{remark}

For a given endomorphism $ I: \mathfrak{g} \to \mathfrak{g} $ with $ I^2 = -\mathrm{Id}_\mathfrak{g} $, we complexify $ \mathfrak{g} $
$$
\mathfrak{g}^{\mathbb{C}} := \mathfrak{g} \otimes_{\mathbb{R}} \mathbb{C} = \mathfrak{g}^{1,0} \oplus \mathfrak{g}^{0,1},
$$
where
$$
\mathfrak{g}^{1,0} = \{ X - \sqrt{-1}IX \mid X \in \mathfrak{g} \}, \quad
\mathfrak{g}^{0,1} = \{ X + \sqrt{-1}IX \mid X \in \mathfrak{g} \}
$$
are the $ +\sqrt{-1} $ and $ -\sqrt{-1} $ eigenspaces of the $ \mathbb{C} $-linear extension of $ I $ to $ \mathfrak{g}^{\mathbb{C}} $, respectively.
The integrability condition $ N_I \equiv  0 $ is equivalent to the requirement that $ \mathfrak{g}^{1,0} $ forms a complex Lie subalgebra of $ \mathfrak{g}^{\mathbb{C}} $
$$
[\mathfrak{g}^{1,0}, \mathfrak{g}^{1,0}] \subseteq \mathfrak{g}^{1,0}.
$$

Let $\sigma: \mathfrak{g}^{\mathbb{C}} \to \mathfrak{g}^{\mathbb{C}}$ denote the conjugation map with respect to the real form $\mathfrak{g}$, i.e.,
$$
\sigma(X - \sqrt{-1}IX) = X + \sqrt{-1}IX, \qquad \sigma(X + \sqrt{-1}IX) = X - \sqrt{-1}IX.
$$
Then $\sigma(\mathfrak{g}^{1,0}) = \mathfrak{g}^{0,1}$ and $\sigma(\mathfrak{g}^{0,1}) = \mathfrak{g}^{1,0}$.

Let $ \{ E_1, \dots, E_n, IE_1,\dots,IE_n \} $ be a basis of $ \mathfrak{g} $ such that $\{\mathcal{E}_i = E_i - \sqrt{-1}IE_i\}$ is a basis of $\mathfrak{g}^{1,0}$.
Extend $\iota$ complex linearly to $\iota_{\mathbb{C}}: \mathfrak{g}^{\mathbb{C}} \to \mathcal{X}_{\mathbb{C}}(G)$,
where $\mathcal{X}_{\mathbb{C}}(G)$ denotes the space of complexified smooth vector fields $\mathcal{X}(G)$ on $G$, by setting
$$
\iota_{\mathbb{C}}(X - \sqrt{-1}Y) := \iota(X) - \sqrt{-1}\iota(Y), \quad
\iota_{\mathbb{C}}(X + \sqrt{-1}Y) := \iota(X) + \sqrt{-1}\iota(Y),
$$
for any $X,Y \in \mathfrak{g}$.
Then define
$$
e_i := \iota_{\mathbb{C}}(\mathcal{E}_i) = \iota(E_i) - \sqrt{-1}\iota(IE_i).
$$
For any $X\in\mathfrak{g}$, since $J$ is left-invariant and $\iota(X)$ is left-invariant, the vector field $J(\iota(X))$ is also left-invariant. Its value at any $g\in G$ is given by
$$
(J(\iota(X)))_g = (dL_g)_e J_e(\iota(X)_e) = (dL_g)_e I(X) = \iota(I(X))_g,
$$
hence $J(\iota(X))=\iota(I(X))$. Extending $J$ complex linearly gives $J\circ\iota_{\mathbb C}=\iota_{\mathbb C}\circ I$ on $\mathfrak g^{\mathbb C}$. Since $\mathcal E_i\in\mathfrak g^{1,0}$, we have $I\mathcal E_i=\sqrt{-1}\mathcal E_i$. Then
$$
J e_i = J(\iota_{\mathbb C}(\mathcal E_i))=\iota_{\mathbb C}(I\mathcal E_i)
=\iota_{\mathbb C}(\sqrt{-1}\mathcal E_i)=\sqrt{-1}e_i,
$$
Thus $e_i$ is a smooth left-invariant complex vector field of type $(1,0)$ on $G$, and $\{e_1,\dots,e_n\}$ is a global smooth frame of $T^{1,0}G$.
Similarly,
$$
\bar e_i := \iota_{\mathbb{C}}(\sigma(\mathcal{E}_i))= \iota(E_i) + \sqrt{-1}\iota(IE_i)
$$
is a smooth left-invariant complex vector field of type $(0,1)$ on $G$, and $\{\bar e_1,\dots,\bar e_n\}$ is a global smooth frame of $T^{0,1}G$.
Let $\{\varphi^1,\dots,\varphi^n\}$ be the left-invariant complex $(1,0)$-forms dual to $\{e_i\}$, i.e., $\varphi^i(e_j) = \delta^i_j$.
\begin{remark}
For a general left-invariant complex structure $J$, the left-invariant complex vector fields $e_i$ are smooth and of type $(1,0)$, but they are not necessarily holomorphic. Indeed, the flow of the real vector field $\iota(X)$ is the right translation $R_{\exp(tX)}$. Thus $\iota(X)$ is holomorphic if and only if the right translations generated by $\iota(X)$ are holomorphic, which is equivalent to $J$ being right-invariant. Since $J$ is already left-invariant, this means $J$ is bi-invariant, i.e., $G$ is a complex Lie group. Therefore, the frame $\{e_i\}$ is holomorphic precisely when $G$ is a complex Lie group.
\end{remark}

Define the functions $\lambda^i_{\,jk}, \lambda^i_{\,j\bar{k}}, \lambda^{\bar{i}}_{\,j\bar{k}}$ on $G$ by the following relations
\begin{equation}\label{eq:4_1}
    [e_j, e_k] = \lambda^i_{\,jk} \, e_i, \qquad
    [e_j, \bar{e}_k] = \lambda^i_{\,j\bar{k}} \, e_i + \lambda^{\bar{i}}_{\,j\bar{k}} \, \bar{e}_i.
\end{equation}
Since the vector fields $e_i$ and $\bar{e}_i$ are left-invariant, their Lie brackets are also left-invariant.
Consequently, the coefficients $\lambda^i_{\,jk}, \lambda^i_{\,j\bar{k}}, \lambda^{\bar{i}}_{\,j\bar{k}}$ must be constant on $G$.
At the identity $e \in G$, the extended isomorphism $\iota_{\mathbb{C}} : \mathfrak{g}^{\mathbb{C}} \to \mathcal{X}_{\mathbb{C}}(G)$ satisfies
$$
[e_j, e_k]_e = \iota_{\mathbb{C}} \bigl([\mathcal{E}_j, \mathcal{E}_k]\bigr)_e,
\qquad
[e_j, \bar{e}_k]_e = \iota_{\mathbb{C}} \bigl([\mathcal{E}_j, \sigma(\mathcal{E}_k)]\bigr)_e,
$$
where $[\mathcal{E}_j, \mathcal{E}_k] = c^i_{\,jk} \mathcal{E}_i$ and $[\mathcal{E}_j, \sigma(\mathcal{E}_k)] = c^i_{\,j\bar{k}} \mathcal{E}_i + c^{\bar{i}}_{\,j\bar{k}} \sigma(\mathcal{E}_i)$ are the Lie brackets in $\mathfrak{g}^{\mathbb{C}}$.
Evaluating \eqref{eq:4_1} at $e$ yields
$$
\lambda^i_{\,jk} = c^i_{\,jk}, \quad
\lambda^i_{\,j\bar{k}} = c^i_{\,j\bar{k}}, \quad
\lambda^{\bar{i}}_{\,j\bar{k}} = c^{\bar{i}}_{\,j\bar{k}}.
$$

\begin{example}\label{ch2} We provide an example of a real Lie group of dimension $4$ equipped with a left-invariant complex structure.
According to a classical result of Heintze \cite{Heintze}, every connected homogeneous Riemannian manifold of nonpositive sectional curvature can be represented as a
connected solvable Lie group $ S $ equipped with a left-invariant Riemannian metric. In particular, for symmetric spaces of negative curvature, this realization
can be explicitly obtained via the Iwasawa decomposition of their isometry group. More precisely, taking the complex hyperbolic plane $ \mathbb{C}H^2 $ as an example, its isometry group is $ SU(1,2) $, whose Iwasawa decomposition is given by \cite{Heintze}
$$
SU(1,2) = K \cdot A \cdot N,
$$
where $ K \cong U(2) $ is a maximal compact subgroup, $ A \cong \mathbb{R} $ is a one-dimensional abelian subgroup, and $ N $ is the three-dimensional Heisenberg group $ H_3 $.
The solvable subgroup $ S = A\cdot N $ acts simply transitively on $ \mathbb{C}H^2 $, allowing $ \mathbb{C}H^2 $ to be identified with $ S $ as a smooth manifold.
The corresponding Lie algebra $ \mathfrak{ch}_2 = \mathrm{Lie}(S) $ is spanned by four generators $ X, Y, Z, W $, with non-zero Lie brackets
\begin{equation*}
	[X, Y] = \tfrac12 Y, \quad [X, Z] = \tfrac12 Z, \quad [X, W] = W, \quad [Z, Y] = W.
\end{equation*}
This Lie algebra is a semidirect product of the Heisenberg ideal $ \langle Y, Z, W \rangle $ and the one-dimensional subalgebra $ \langle X \rangle $.

Define a real linear map $I : \mathfrak{ch}_2 \to \mathfrak{ch}_2$   by
$$
IX = \sqrt{\tfrac{\gamma}{\beta}} W,\quad IY = -Z, \quad IZ = Y, \quad IW = -\sqrt{\tfrac{\beta}{\gamma}} X,
$$
where $\beta, \gamma > 0$ are parameters. One easily verifies that $I^2 = -\mathrm{Id}$.

A direct computation of the Nijenhuis tensor $N_I(X,Y)$
for all basis vectors shows that $N_I \equiv 0$. Hence the left-invariant almost complex structure $J$ induced by $I$ is integrable, so $\mathbb{C}H^2$ is a complex manifold.

In the complexified Lie algebra $\mathfrak{ch}_2^{\mathbb{C}}$, consider the eigenspaces
\begin{equation*}
	\begin{split}
		&\mathfrak{ch}_2^{1,0} = \mathrm{span}\left\{ e_1 = \frac{1}{\sqrt{2\gamma}}(X - \sqrt{-1}IX),\ e_2 = \frac{1}{\sqrt{2}}(Y -  \sqrt{-1}IY) \right\},\\
		&\mathfrak{ch}_2^{0,1} = \mathrm{span}\left\{\bar e_1 =\frac{1}{\sqrt{2\gamma}}(X +  \sqrt{-1}IX),\ \bar e_2= \frac{1}{\sqrt{2}}(Y + \sqrt{-1}IY) \right\}.
	\end{split}
\end{equation*}
All nonzero Lie brackets of $\mathfrak{ch}_2^{\mathbb{C}}$ are as follows
\begin{align*}
&[e_1,e_2]=\frac{1}{2\sqrt{2\gamma}}e_2,\ [e_1,\bar e_1]=\frac{1}{\sqrt{2\gamma}}\bar e_1-\frac{1}{\sqrt{2\gamma}}e_1,\\
&[e_1,\bar e_2]=\frac{1}{2\sqrt{2\gamma}}\bar e_2,\ [e_2,\bar e_2]=\frac{\sqrt{\beta}}{\sqrt{2}}\bar e_1- \frac{\sqrt{\beta}}{\sqrt{2}} e_1.
\end{align*}
Thus we verified that $[\mathfrak{ch}_2^{1,0}, \mathfrak{ch}_2^{1,0}] \subset \mathfrak{ch}_2^{1,0}$.

\end{example}

\subsection{Left-invariant complex Finsler metric on Lie groups}
\hspace{1em}
From now on, let $G$ be a real Lie group which is also a complex manifold. Denote by $\mathfrak{g}$ the real Lie algebra of $G$ and $\mathfrak{g}^{1,0}$ the $(1,0)$-part of $\mathfrak{g}^{\mathbb{C}}=\mathfrak{g}\otimes_{\mathbb{R}} \mathbb{C}$. In this section, we investigate the relationship between the coefficients of the complex Rund connection and the structure constants of $\mathfrak{g}^{\mathbb{C}}$. Let $v$  be a smooth left-invariant  complex vector field on $G$ and $v_g\in T_g^{1,0}G, v_e\in \mathfrak g^{1,0}$ be its values at $g\in G$ and $e\in G$, respectively,  so that $v_g = (dL_g)_e (v_e)$, or equivalently $v_e=(dL_{g^{-1}})_g(v_g)$.

\begin{proposition}\label{Po:4_5}
Let $v_e = v^i\mathcal E_i \in \mathfrak g^{1,0}$ with $v^i=y^i+\sqrt{-1}y^{i+n} \in \mathbb C$, and let $L_g$ be the left translation. Extending $(dL_g)_e$ in a natural way to $T_g^{1,0}G$, one has
$$
(dL_g)_e(v_e) = v^i e_i(g).
$$
\end{proposition}
\begin{proof}
Indeed,
\begin{eqnarray*}
(dL_g)_e(v_e):&=&(dL_g)_e\left((y^i+\sqrt{-1}y^{i+n})e_i(e)\right)\\
&=&(dL_g)_e\left((y^i+\sqrt{-1}y^{i+n})\mathcal E_i\right)\\
&=&(dL_g)_e\left((y^i+\sqrt{-1}y^{i+n})\left(E_i-\sqrt{-1}IE_i\right)\right)\\
&=&(dL_g)_e\left(y^i E_i+y^{i+n} IE_i\right)+\sqrt{-1}\left((dL_g)_e(y^{i+n}E_i-y^i IE_i)\right)\\
&=&y^i(dL_g)_e(E_i)+y^{i+n}(dL_g)_e(IE_i)\\
&&+y^{i+n}\sqrt{-1}(dL_g)_e(E_i)-y^i\sqrt{-1}(dL_g)_e(IE_i)\\
&=&y^i\iota(E_i)_g+y^{i+n}\iota(IE_i)_g+y^{i+n}\sqrt{-1}\iota(E_i)_g-y^i\sqrt{-1}\iota(IE_i)_g\\
&=&(y^i+\sqrt{-1}y^{i+n})\iota(E_i)_g+(y^{i+n}-\sqrt{-1}y^i)\iota(IE_i)_g\\
&=&(y^i+\sqrt{-1}y^{i+n})\left(\iota(E_i)_g-\sqrt{-1}\iota(IE_i)_g\right)\\
&=&v^ie_i(g).
\end{eqnarray*}
\end{proof}

\begin{remark}
 Via the identification $\iota_{\mathbb C}$, an element $v\in\mathfrak g^{1,0}$ is naturally identified with its corresponding left-invariant complex vector field--its value at $g\in G$ is $v_g=(dL_g)_e (v_e)=v^ie_i(g)$. Thus in the Lie group case, the same symbol $v$ stands for both the element of $\mathfrak{g}^{1,0}$ and its associated left-invariant complex vector field, which always corresponds to the radial vector with variable direction---also denoted by $v=(v^1,\cdots,v^n)$ if it causes no confusion. In the following, we will not distinguish among $v\in\mathfrak{g}^{1,0}$, the left-invariant complex vector field determined by $v$, and the radial vector with variable direction $v=(v^1,\cdots,v^n)$. Sometimes, we just call $v$ \emph{the radial vector with variable direction}.
 \end{remark}

Let $F:T^{1,0}G\rightarrow [0,+\infty)$ be a left-invariant complex Finsler metric on a Lie group $G$, which is also a complex manifold. By left-invariance and Proposition \ref{Po:4_5}, for any $g\in G$ and $v_g\in T_g^{1,0}G$ (which is the value of a left-invariant complex vector field $v$), the metric $F$ satisfies
$$
F(g; v_g) = F(L_g(e); (dL_{g})_e(v_e))=F(e;v_e),
$$
where $e$ denotes the identity element of $G$ and $v_e\in \mathfrak{g}^{1,0}$---the value of $v$ at $e$.
Therefore, for a left-invariant complex vector field $v$ and any $g\in G$, the function
$$
G\ni g \;\longmapsto\; F(g; v_g)=F(e;v_e)
$$
is constant (equal to $\|v_e\| := F(e; v_e)$), namely it does not depend on the base point $g$. Hence on the punctured holomorphic tangent bundle  $\tilde{G}=T^{1,0}G\setminus\{\mbox{zero section}\}$, the total differential of $F(g;v_g)$ has only fiber derivatives:
$$
dF = e_i(F)\,\varphi^i +\bar{e}_i(F)\,\bar{\varphi}^i+ \frac{\partial F}{\partial v^i}\,dv^i+\frac{\partial F}{\partial\bar{v}^i}d\bar{v}^i= \frac{\partial F}{\partial v^i}\,dv^i+\frac{\partial F}{\partial\bar{v}^i}d\bar{v}^i,
$$
where $(v^i)$ are the fiber coordinates relative to the global left-invariant smooth frame $\{e_i\}$ for $T^{1,0}G$ (so that $v=v^ie_i$).
This shows that a left-invariant complex Finsler metric $F$ on $G$ is in bijective correspondence with a complex Minkowski norm on the complex Lie subalgebra $\mathfrak{g}^{1,0}$. Consequently, $F$ can be  viewed as a function solely of
the fiber coordinates $v^i$ and $\bar{v}^i$.

\begin{proposition}\label{Po:4_7}
Let $G$ be an even-dimensional connected real Lie group which is also a complex manifold. Let $\mathfrak g$ denote the Lie algebra of $G$, and $J$ a left-invariant complex structure on $G$. Then there is a one-to-one correspondence between left-invariant complex Finsler metrics on $ G $ and complex Minkowski norms  on $ \mathfrak{g}^{1,0} $.
\end{proposition}
 Because  $F:T^{1,0}G\rightarrow [0,+\infty)$ on $G$  is considered left-invariant, all geometric quantities are constant with respect to the global left-invariant frame $\{e_i\}$ on $G$. Hence the complex Rund connection 1-forms $\omega^i_j$,
 originally defined on $\tilde{G}$, depend only on fiber coordinates $v \in  \mathfrak{g}^{1,0}$.
 We may therefore regard the corresponding geometric quantities as functions on $\mathfrak{g}^{1,0}$ (or on $\mathfrak{g}^{1,0}\setminus\{0\}$).
 In this sense, the connection $1$-forms $\omega^i_j$ restrict to left-invariant forms on $\tilde{G}$.

Since
\begin{equation*}
	\begin{split}
		&d\varphi^i(e_j, e_k)=e_j(\varphi^i(e_k))- e_k(\varphi^i(e_j))-\varphi^i([e_j, e_k])=-\lambda^i_{\ jk},\\
		&d\varphi^i(e_j, \bar{e}_k)=e_j(\varphi^i(\bar{e}_k))- \bar{e}_k(\varphi^i(e_j))-\varphi^i([e_j, \bar{e}_k])=-\lambda^i_{\ j\bar{k}},
	\end{split}
\end{equation*}
it follows that
\begin{equation}\label{eq:4_2}
	d\varphi^i=-\frac{1}{2}\lambda^i_{\ j k}\varphi^j\wedge \varphi^k-\lambda ^i _{\ j\bar k}\varphi^j\wedge \bar \varphi^k.
\end{equation}

Furthermore, under the above identifications, the fundamental tensor $g_{i\bar{j}}$ and the complex Cartan tensors $C_{i\bar{j}k}$, $C_{i\bar{j}\bar{k}}$ are calculated explicitly as follows
 \begin{equation*}
	g_{i\bar{j}}=[F^2]_{v^i\bar v^j};\;
	C_{i\bar{j} k}=[F^2]_{v^i\bar v^j v^k}; \; C_{i\bar{j}\bar k}=[F^2]_{v^i\bar v^j\bar v^k}.
 \end{equation*}

Substituting $\omega^i_j=\Gamma^i_{\;jk}\varphi^k+\Gamma^i_{\;j\bar{k}}\varphi^k$ and $\tau^i=T^i_{\;jk}\varphi^j\wedge\varphi^k$ into \eqref{eq:3_1} and using \eqref{eq:4_2}, we have
\begin{align*}
	&\varphi ^j\wedge (\Gamma^i_{\ jk}\varphi^k+\Gamma^i_{\ j \bar k }\bar \varphi^k)+T^i_{\ jk}\varphi^j\wedge \varphi^k=
	-\frac{1}{2}\lambda^i_{\ j k}\varphi^j\wedge \varphi^k-\lambda ^i _{\ j\bar k}\varphi^j\wedge \bar \varphi^k.
\end{align*}
Thus,
\begin{eqnarray}
	\varphi ^j\wedge \Gamma^i_{\ jk}\varphi^k+T^i_{\ jk}\varphi^j\wedge \varphi^k&=&-\frac{1}{2}\lambda^i_{\ j k}\varphi^j\wedge \varphi^k;\label{eq:4_3}\\
	\varphi^j \wedge \Gamma^i_{\ j \bar k }\bar\varphi^k&=&-\lambda^i_{\ j\bar k}\varphi^j\wedge \bar\varphi^k.\label{eq:4_4}
\end{eqnarray}
Using \eqref{eq:4_3} and the Cartan lemma yields
\begin{align*}
	\Gamma^i_{\ jk}+T^i_{\ jk}+\frac{1}{2}\lambda^i_{\ j k}=S^i_{\ jk},
\end{align*}
where $S^i_{\ jk}=S^i_{\ kj}$ are functions defined on $G$.
Consequently, the following equalities hold
\begin{align}
	&\Gamma^i_{\ jk}-\Gamma^i_{\ kj}=-2T^i_{\ jk}-\lambda^i_{\ jk}; \label{eq:4_5} \\
	&\Gamma^i_{\ jk}+\Gamma^i_{\ kj}=2\Gamma^i_{\ jk}+2T^i_{\ jk}+\lambda^i_{\ jk}. \label{eq:4_6}
\end{align}
Using \eqref{eq:4_4}, one directly gets $\Gamma^i_{\ j\bar k}=-\lambda^i_{\ j\bar k}$.
\par
Next, using \eqref{eq:3_2}, we find
\begin{align*}
	\frac{\partial  g_{i\bar{j}}}{\partial   v^l}d  v^l+\frac{\partial  g_{i\bar{j}}}{\partial \bar v^l}d\bar v^l
                =&g_{l\bar{j}}(\Gamma ^l_{\ ik}\varphi^k+\Gamma ^l_{\ i \bar k} \bar \varphi^k)+g_{i\bar l}(\Gamma ^{\bar l}_{\ \bar{j} \bar k}\bar \varphi^k+\Gamma ^{\bar l}_{\ \bar{j} k}\varphi^k)\\
				&+C_{i\bar{j} l}[d  v^l+  v^s(\Gamma ^l_{\ sk}\varphi^k+\Gamma ^l_{\ s \bar k} \bar \varphi^k)]\\
				&+C_{i\bar{j}\bar s}[d\bar v^s+\bar v^q(\Gamma ^{\bar s}_{\ \bar q \bar k}\bar \varphi ^k+\Gamma ^{\bar s}_{\ \bar q k}\varphi^k)].
\end{align*}
Comparing their type on both sides of the above equation yields
\begin{align}
    0=g_{l\bar{j}}\Gamma ^l_{\ ik}+g_{i\bar l}\Gamma ^{\bar l}_{\ \bar{j} k}
				+C_{i\bar{j} l}  v^s\Gamma ^l_{\ sk}+
				C_{i\bar{j}\bar s}\bar v^q\Gamma ^{\bar s}_{\ \bar q k}. \label{eq:4_7}
\end{align}
Exchanging the indices $i$ and $k$ in the above equation gives
\begin{equation}\label{eq:4_8}
    0=g_{l\bar{j}}\Gamma ^l_{\ ki}+g_{k\bar s}\Gamma ^{\bar s}_{\ \bar{j} i}
				+C_{k\bar{j} l}  v^s\Gamma ^l_{\ si}+
				C_{k\bar{j}\bar s}\bar  v^q\Gamma ^{\bar s}_{\ \bar q i}.
\end{equation}
Adding \eqref{eq:4_7} and \eqref{eq:4_8} together yields
\begin{equation}
	\begin{split}
	0=&g_{l\bar{j}}(\Gamma ^l_{\ ik}+\Gamma ^l_{\ ki})+g_{i\bar s}\Gamma ^{\bar s}_{\ \bar{j} k}
		+g_{k\bar s}\Gamma ^{\bar s}_{\ \bar{j} i}\\
	&	+C_{i\bar{j} l}  v^s\Gamma ^l_{\ sk}+
		C_{i\bar{j}\bar s}\bar v^q\Gamma ^{\bar s}_{\ \bar q k}
		+C_{k\bar{j} l}  v^s\Gamma ^l_{\ si}+
		C_{k\bar{j}\bar s}\bar  v^q\Gamma ^{\bar s}_{\ \bar q i}.
	\end{split}
\end{equation}
Substituting $\Gamma^i_{\;j \bar k}=-\lambda^i_{\ j\bar k}$ and $\Gamma^{\bar{i}}_{\ \bar{j} k}=-\lambda^{\bar{i}}_{\;\bar{j} k}$ and applying \eqref{eq:4_6} then gives
\begin{equation}\label{eq:4_10}
	\begin{split}
    0=&2g_{l\bar{j}}\Gamma ^l_{\ ik}+2g_{l\bar{j}}T^l_{\ ik}+g_{l\bar{j}}\lambda^l_{\ ik}-g_{i\bar s }\lambda^{\bar s}_{\ \bar{j} k}-g_{k \bar s}\lambda^{\bar s}_{\ \bar{j} i}\\
	&+C_{i \bar{j} l}  v^s \Gamma^l_{\ sk}-C_{i\bar{j} \bar s}\bar  v^q\lambda^{\bar s}_{\ \bar q k}+
	C_{k \bar{j} l}  v^s\Gamma^l_{\ s i}-C_{k \bar{j} \bar s}\bar   v^q\lambda^{\bar s}_{\ \bar q i}.
	\end{split}
\end{equation}
Now subtracting \eqref{eq:4_8} from \eqref{eq:4_7}, we obtain
\begin{equation*}
	\begin{split}
			0=&g_{l\bar{j}}(\Gamma ^l_{\ ik}-\Gamma ^l_{\ ki})+g_{i\bar s}\Gamma ^{\bar s}_{\ \bar{j} k}-g_{k\bar s}\Gamma ^{\bar s}_{\ \bar{j} i}\\
				&+C_{i\bar{j} l}  v^s\Gamma ^l_{\ sk}
				+C_{i\bar{j}\bar s}\bar v^q\Gamma ^{\bar s}_{\ \bar q k}
				-C_{k\bar{j} l}  v^s\Gamma ^l_{\ si}
				-C_{k\bar{j}\bar s}\bar  v^q\Gamma ^{\bar s}_{\ \bar q i}.
	\end{split}
\end{equation*}
Substituting $\Gamma^i_{\;j \bar k}=-\lambda^i_{\ j\bar k}$ and $\Gamma^{\bar{i}}_{\ \bar{j} k}=-\lambda^{\bar{i}}_{\;\bar{j} k}$ into the above equation and then using \eqref{eq:4_5} gives
\begin{align*}
	0&=g_{l \bar{j}}(-2T^l_{\ ik}-\lambda^l_{\ ik})-g_{i\bar s}\lambda^{\bar s}_{\ \bar{j} k}+g_{k \bar s}\lambda^{\bar s}_{\ \bar{j} i}\nonumber\\
		&\ \ \ +C_{i\bar{j} l}  v^s\Gamma^l_{\ sk}-C_{i\bar{j} \bar s}\bar  v^q\lambda^{\bar s}_{\ \bar q k}-C_{k\bar{j} l}  v^s\Gamma^l_{\ si}+
		C_{k\bar{j} \bar s}\bar   v^q\lambda^{\bar s}_{\ \bar q i}.
\end{align*}
Hence,
\begin{align}\label{eq:4_11}
	2g_{s \bar{j}}T^s_{\ ik}&=-g_{l\bar{j}}\lambda^l_{\ ik}-g_{i\bar s}\lambda^{\bar s}_{\ \bar{j} k}+g_{k \bar s}\lambda^{\bar s}_{\ \bar{j} i} \nonumber \\
		&\ \ \ +C_{i\bar{j} l}  v^s\Gamma^l_{\ sk}-C_{i\bar{j} \bar q}\bar  v^s\lambda^{\bar q}_{\ \bar s k}-C_{k\bar{j} l}  v^s\Gamma^l_{\ si}+
		C_{k\bar{j} \bar s}\bar   v^q\lambda^{\bar s}_{\ \bar q i}.
\end{align}
Substituting \eqref{eq:4_11} into \eqref{eq:4_10} leads to
\begin{equation}
	2g_{s\bar{j}}\Gamma^s_{\ ik}=2g_{i \bar s}\lambda^{\bar s}_{\ \bar{j} k}-2C_{i \bar{j} l}  v^s\Gamma^l_{\ sk}+2C_{i\bar{j}\bar s}\bar  v^q \lambda^{\bar s}_{\ \bar q k}.\label{eq:4_12}
\end{equation}

Now we define
$$
g_{ij} := \frac{\partial^2 F^2}{\partial v^i\partial v^j},\qquad
g_{\bar i\bar j} := \frac{\partial^2 F^2}{\partial \bar v^i\partial \bar v^j}.
$$
It is clear that $g_{\bar i\bar j} = \overline{g_{ij}}$. Note that $g_{\bar j\bar s}$ can also be obtained by contraction with the complex Cartan tensor, namely
$$
g_{\bar j\bar s} = C_{i\bar j\bar s} v^i.
$$
Let $v\in\mathfrak g^{1,0}\setminus\{0\}$ be the radial vector with variable direction and vectors $u,w\in\mathfrak g^{1,0}$. We define the Hermitian inner product $g_v(u,w)$, the symmetric product $\mathcal{S}_v(u,w)$ and the conjugate symmetric product $\overline{\mathcal{S}}_v(u,w)$, respectively as follows:
\begin{equation}
g_v(u,w):= g_{i\bar j}u^i\bar w^j,\quad
\mathcal{S}_v(u,w):= g_{ij}u^i w^j,\quad \overline{\mathcal{S}}_v(u,w):= g_{\bar{i}\bar{j}}\bar{u}^i \bar{w}^j.\label{eq:4_13}
\end{equation}

Contracting \eqref{eq:4_12} with $  v^i$ yields
\begin{equation*}
	g_{s\bar{j}}N^s_{\;k}=\lambda_{o\bar{j} k}+g_{\bar{j}\bar s}\lambda^{\bar s}_{\ \bar o k},
\end{equation*}
where $N^i_{\;j}=  v^s\Gamma^i_{\ s j}$, $C_{i\bar{j}\bar s}v^i=g_{\bar{j}\bar s}$,  $g_{i\bar{s}}\lambda^{\bar{s}}_{\ \bar{j} k}v^i=\lambda_{o\bar{j} k}$
and $\lambda^{\bar s}_{\ \bar q k}\bar v^q=\lambda^{\bar s}_{\ \bar o k}$.

To summarize, the above computations give
\begin{align}
	\Gamma^j_{\ ik}&=\lambda^{\ j}_{i\  k}-C^{\ j}_{i \ l}N^l_{\;k}+C^{\ j}_{i\ \bar s}\lambda^{\bar s}_{\ \bar ok},
	\quad \Gamma^j_{\ i\bar k}=-\lambda^j_{\ i\bar k}; \label{eq:4_14}   \\
	T^j_{\ ik}&=\frac{1}{2}\{\lambda^{\ j}_{k\  i}-\lambda^j_{\ ik}-\lambda^{\ j}_{i \  k}
		+C^{\ j}_{i\  l}N^l_{\;k}-C^{\ j}_{i\  \bar s}\lambda^{\bar s}_{\ \bar o k}-C^{\ j}_{k\  l}N^l_{\;i}+
		C^{\ j}_{k\  \bar s}\lambda^{\bar s}_{\ \bar o i}\} \label{eq:4_15},
\end{align}
where
\begin{equation}
	N^j_{\;k}=\lambda^{\ j}_{o \  k}+g^{\bar q j}g_{\bar q\bar s}\lambda^{\bar s}_{\ \bar o k}.\label{eq:4_16}
\end{equation}

\begin{definition}\label{De:4_8}
The connection operator $\mathcal N_v:\mathfrak g^{1,0}\to\mathfrak g^{1,0}$ is defined by
$$
\mathcal N_v(w):= N_{\;j}^i(v) w^j \mathcal{E}_i,\quad \forall w=w^i\mathcal{E}_i\in\mathfrak{g}^{1,0},
$$
where the coefficients $N_{\;j}^i(v)$ are given by \eqref{eq:4_16}.
\end{definition}

\begin{remark}
If $F$ comes from a Hermitian metric, the complex Cartan tensors vanish identically ($C_{i\bar{j}k} = C_{i\bar{j}\bar{k}} = 0$),
and the complex nonlinear connection coefficients reduce to $N_{\;j}^i = \Gamma_{\;kj}^i v^k$. In this case, the torsion tensor reduces to
$$
T^i_{\ jk} = \frac{1}{2}\{\lambda^{\ j}_{k\  i}-\lambda^j_{\ ik}-\lambda^{\ j}_{i \  k}\}.
$$
This is precisely the torsion of the Chern connection in Hermitian geometry (see equation (4) on page 601 in \cite{Vezzoni-Yang-Zheng}).
Our construction thus naturally generalizes the classical Chern connection in Hermitian geometry to complex Finsler geometry.
\end{remark}

\begin{proposition}
The operator $\mathcal N_v:\mathfrak g^{1,0}\to\mathfrak g^{1,0}$ defined by Definition \ref{De:4_8} satisfies the invariant identity
$$
g_v(\mathcal N_v(w),u) = g_v(v,[u,\sigma(w)]^{1,0}) + \overline{\mathcal{S}}_v(u,[v,\sigma(w)]^{1,0})
$$
for all $u,w\in\mathfrak g^{1,0}$.
\end{proposition}

Setting $N^i_{\;\bar{j}}:=-\lambda^i_{\ k\bar{j}}v^k$, and extending $\mathcal{N}_v$  to $\mathfrak{g}^{\mathbb{C}}$ by setting $\mathcal{N}_v(w)=N^i_{\;j}w^j\mathcal{E}_i$ and $\mathcal{N}_v(\sigma(w))=N^i_{\;\bar{j}}\bar w^j\mathcal{E}_i$, we obtain a complex  linear operator $\mathcal{N}_v:\mathfrak{g}^{\mathbb{C}}\rightarrow \mathfrak{g}^{\mathbb{C}}$, also called  the connection
operator on $\mathfrak{g}^{\mathbb{C}}$.
If $F$ comes from a Hermitian metric,  then the symmetric product $\mathcal{S}_v\equiv0$,  hence $\mathcal{N}_v$ reduces to the Chern connection of $F$.
For any $u,w,h\in \mathfrak{g}^{1,0}$, define
\begin{equation*}
		\mathcal{C}_v^+(w, u, h):=C_{i\bar j l}w^i\bar{u}^jh^l,\qquad\mathcal{C}_v^-(w, u, h):=C_{i\bar j l}w^i\bar{u}^j\bar{h}^l.
\end{equation*}

Therefore,  according to the Definition \ref{De:3_5}, one has
\begin{theorem} \label{Th:4_11}
Let $G$ be an even-dimensional connected real Lie group which is also a complex manifold. Let $\mathfrak g$ denote the Lie algebra of $G$, and $J$ a left-invariant complex structure on $G$. Let $ F:T^{1,0}G\rightarrow [0,+\infty) $ be a left-invariant complex Finsler metric on $ G $.
Then the connection and torsion coefficients of the complex Rund connection are given by (\ref{eq:4_14}) and (\ref{eq:4_15}), respectively. In particular,  $F$ is a K\"ahler-Finsler metric if and only if
	\begin{align*}
		\lambda_{k \bar q o}-\lambda_{\bar q ok}-\lambda_{o\bar q k}
		-g_{\bar q \bar s}\lambda^{\bar s}_{\ \bar o k}-C_{k\bar q l}N^l_{\;i}v^i+
		C_{k\bar q \bar s}\lambda^{\bar s}_{\ \bar o o}=0
	\end{align*}
	or equivalently, for any  $u,w\in \mathfrak{g}^{1,0}$
	\begin{align*}
	&g_v(w, [ u,  \sigma(v)]^{1, 0})-g_v([v, w]^{1, 0}, u)-g_v(v, [ u, \sigma(w)]^{1, 0})\nonumber\\
	&-\overline{\mathcal{S}}_v(u, [v, \sigma(w)]^{1, 0})-\mathcal{C}_v^+(w, u, \mathcal{N}_v(v))+\mathcal{C}_v^-(w, u, [v, \sigma(v)]^{1, 0})=0,
	\end{align*}
	where $v\in\mathfrak{g}^{1,0}\setminus\{0\}$ is the radial vector with variable direction;
$F$ is a weakly K\"ahler-Finsler metric if and only if
	\begin{align*}
		\lambda_{k\bar o o}-\lambda_{\bar o ok}-\lambda_{o\bar o k}-g_{k l}N^l_{\;i}v^i=0,
	\end{align*}
	or equivalently, for any $w\in \mathfrak{g}^{1,0}$
	\begin{equation*}
		g_v(w, [v, \sigma(v)]^{1, 0})-g_v([v, w]^{1, 0}, v)-g_v(v, [v, \sigma(w)]^{1, 0})-\mathcal{S}_v(\mathcal{N}_v(v), w)=0,
	\end{equation*}
	where $v\in\mathfrak{g}^{1,0}\setminus\{0\}$ is the radial vector with variable direction.
\end{theorem}
\begin{proof}
	According to Definition \ref{De:3_5} for a K\"ahler-Finsler metric, we have $T^j_{\;ik}=0$ where $T^j_{\;ik}$ is given by \eqref{eq:4_15}.
	Chen and Shen  proved that $T^j_{\;ik}=0$ is equivalent to $T^j_{\;ik}v^j=0$ in \cite{ChenAndShen},
	so
			\begin{align*}
		\{g_{k \bar s}\lambda^{\bar s}_{\ \bar q i}&-g_{l\bar q}\lambda^l_{\ ik}-g_{i\bar s}\lambda^{\bar s}_{\ \bar q k}
		-C_{i\bar q \bar l}\bar  v^s\lambda^{\bar l}_{\ \bar s k}-C_{k\bar q l}N^l_{\;i}+
		C_{k\bar q \bar l}\bar   v^s\lambda^{\bar l}_{\ \bar s i}\}v^i=0.
	\end{align*}
	Now  contracting the equation with arbitrary vectors $w^k,\bar u^q$ and using Definition \ref{De:4_8} gives
	\begin{align*}
	&g_v(w, [ u,  \sigma(v)]^{1, 0})-g_v([v, w]^{1, 0}, u)-g_v(v, [ u, \sigma(w)]^{1, 0})\nonumber\\
	&-\overline{\mathcal{S}}_v(u, [v, \sigma(w)]^{1, 0})-\mathcal{C}_v^+(w, u, \mathcal{N}_v(v))+\mathcal{C}_v^-(w, u, [v, \sigma(v)]^{1, 0})=0.
	\end{align*}

	From Definition \ref{De:3_5} for a weakly K\"ahler-Finsler metric, we have $g_{j\bar l}T^j_{\;ik}v^i\bar v^l =0$ which is equivalent to
	\begin{equation*}
		\{g_{k \bar s}\lambda^{\bar s}_{\ \bar q i}-g_{l\bar q}\lambda^l_{\ ik}-g_{i\bar s}\lambda^{\bar s}_{\ \bar q k}-C_{k\bar q l}N^l_{\;i}\}v^i\bar v^{q}=0.
	\end{equation*}
	Similarly,  contracting the above equation with arbitrary vectors $w^k$ yields
	\begin{equation*}
		g_v(w, [v, \sigma(v)]^{1, 0})-g_v([v, w]^{1, 0}, v)-g_v(v, [v, \sigma(w)]^{1, 0})-\mathcal{S}_v(\mathcal{N}_v(v), w)=0.
	\end{equation*}

	This completes the proof.
\end{proof}

As explained in Proposition \ref{Po:4_7} and the subsequent discussion,
a left-invariant complex Finsler metric $F$ on $G$ corresponds to a complex Minkowski norm on $\mathfrak{g}^{1,0}$,
and all associated geometric quantities depend only on the directional variable $v\in\mathfrak{g}^{1,0}\setminus\{0\}$.
Consequently, the curvature tensor of $F$ descends to a tensor field on $\mathfrak{g}^{1,0}\setminus\{0\}$,
and all curvature formulas can be expressed purely in terms of the Lie algebra structure and the quantities derived from $F$ on $\mathfrak{g}^{1,0}$.

We now introduce a flat connection on the manifold $\mathfrak g^{1,0}\setminus\{0\}$. Let $\{\mathcal E_i\}$ be a basis of $\mathfrak g^{1,0}$. Every point $v\in\mathfrak g^{1,0}\setminus\{0\}$ can be uniquely written as $v=v^i\mathcal E_i$, so that $(v^1,\dots,v^n)$ are global holomorphic coordinates. Identifying $T_v^{1,0}(\mathfrak g^{1,0}\setminus\{0\})$ with $\mathfrak g^{1,0}$ via $\partial_{v^i}\mapsto \mathcal E_i$, we define a flat connection $\mathcal D$ by
\begin{equation}
\mathcal D_w v := w,\qquad \mathcal D_{\sigma(w)} v := 0\label{eq:4_17}
\end{equation}
for any $w\in\mathfrak g^{1,0}$, where $v$ is understood as the radial vector field $v=v^i\partial_{v^i}$ on the coordinate space. In local coordinates, if $w=w^i\partial_{v^i}$, then
$$
\mathcal D_{w^i\partial_{v^i}}\bigl(v^j\partial_{v^j}\bigr)
= w^i\,\partial_{v^i}(v^j)\,\partial_{v^j}
= w^j\partial_{v^j}=w,
$$
while the derivative along $\sigma(w)$ is zero. Hence $\mathcal D$ is the usual directional derivative in the holomorphic directions and trivial in the antiholomorphic directions.

\begin{theorem}\label{Th:4_12}
Let $G$ be an even-dimensional connected real Lie group which is also a complex manifold. Let $\mathfrak g$ denote the Lie algebra of $G$, and $J$ a left-invariant complex structure on $G$. Let $v\in\mathfrak g^{1,0}\setminus\{0\}$ be the radial vector with variable direction and let $w\in\mathfrak g^{1,0}\setminus\{0\}$ be an arbitrary vector.
Then the holomorphic bisectional curvature in  directions $v$ and $w$ is given by
$$
B(v, w) = \frac{g_v\bigl(\mathcal{R}(w, \sigma(w))v, v\bigr)}{g_v(v, v)\, g_v(w, w)},
$$
and the holomorphic sectional curvature in the direction $v$ is
$$
K(v) = 2\frac{g_v\bigl(\mathcal{R}(v, \sigma(v))v, v\bigr)}{g^2_v(v, v)},
$$
where $\mathcal{R}(w,\sigma(w))v$ is given by
\begin{equation*}
\begin{split}
	\mathcal{R}(w, \sigma(w))v=&\mathcal{D}_{\sigma(\mathcal{N}_v(w))}\bigl(\mathcal{N}_v(w)\bigr)-\mathcal{D}_{[v,\sigma(w)]^{1,0}}\bigl(\mathcal{N}_v(w)\bigr)
  -\mathcal{N}_v\bigl([w,\sigma(w)]^{1,0}\bigr)\\
&\quad   +[\mathcal{N}_v(w),\sigma(w)]^{1,0}-[v,[\sigma(w),w]^{0,1}]^{1,0}.
\end{split}
\end{equation*}
\end{theorem}

\begin{proof}
According to Definition \ref{De:4_8},
$
\delta v^i=dv^i+v^j\omega^i_j=dv^i+N^i_{\;k}\varphi^k+N^i_{\;\bar{j}}\bar{\varphi}^j.
$
Now using the formula $\Omega^i_j=d\omega^i_j-\omega^k_j\wedge\omega^i_k$ and $\Omega^i = v^k \Omega^i_k$ yields
\begin{align}
\Omega^i &= d(\delta v^i)-\delta v^s\wedge\omega^i_s\nonumber \\
&= dN^i_{\;s}\wedge\varphi^s+N^i_{\;s}d\varphi^s+dN^i_{\;\bar{s}}\wedge\bar{\varphi}^s+N^i_{\;\bar{s}}d\bar{\varphi}^s
      -\delta v^s\wedge(\Gamma^i_{\;sj}\varphi^j+\Gamma^i_{\;s\bar{j}}\bar{\varphi}^j)\nonumber\\
&= \Bigl\{[N^i_{\;k}]_{v^l}N^l_{\;\bar{j}}+[N^i_{\;k}]_{\bar{v}^l}N^{\bar{l}}_{\;\bar{j}}+N^i_{\;l}\Gamma^l_{\;k\bar{j}}
       -[N^i_{\;\bar{j}}]_{v^l}N^l_{\;k}-[N^i_{\;\bar{j}}]_{\bar{v}^l}N^{\bar{l}}_{\;k}-N^i_{\;\bar{l}}\Gamma^{\bar{l}}_{\;\bar{j}k}\Bigr\}\varphi^k\wedge\bar{\varphi}^j\nonumber\\
&\quad +\Bigl\{N^i_{\;l}T^l_{\;jk}-[N^i_{\;k}]_{v^l}N^l_{\;j}-[N^i_{\;k}]_{\bar{v}^s}N^{\bar{s}}_{\;j}-N^i_{\;l}\Gamma^l_{\;kj}\Bigr\}\varphi^j\wedge\varphi^k\nonumber\\
&\quad +\Bigl\{N^i_{\;\bar{s}}T^{\bar{s}}_{\;\bar{j}\bar{k}}-[N^i_{\;\bar{k}}]_{v^l}N^l_{\;\bar{j}}-[N^i_{\;\bar{k}}]_{\bar{v}^s}N^{\bar{s}}_{\;\bar{j}}
        -N^i_{\;\bar{s}}\Gamma^{\bar{s}}_{\;\bar{k}\bar{j}}\Bigr\}\bar{\varphi}^j\wedge\bar{\varphi}^k\nonumber\\
&\quad +\Bigl\{[N^i_{\;k}]_{v^l}-\Gamma^i_{\;lk}\Bigr\}\delta v^l\wedge\varphi^k
        +\Bigl\{[N^i_{\;\bar{k}}]_{v^l}-\Gamma^i_{\;l\bar{k}}\Bigr\}\delta v^l\wedge\bar{\varphi}^k\nonumber\\
&\quad +[N^i_{\;k}]_{\bar{v}^l}\delta \bar{v}^l\wedge\varphi^k
        +[N^i_{\;\bar{k}}]_{\bar{v}^l}\delta \bar{v}^l\wedge\bar{\varphi}^k,\label{eq:4_18}
\end{align}
where $[\cdot]_{v^l}, [\cdot]_{\bar{v}^l}$ denote the derivatives of $[\cdot]$ with respect to $v^l$ and $\bar{v}^l$, etc.
Since $N^i_{\;\bar{j}}=\Gamma^i_{\;s\bar{j}}v^s=-\lambda^i_{\;s\bar{j}}v^s$,
$$
[N^i_{\;\bar{j}}]_{v^s}=\Gamma^i_{\;s\bar{j}}=-\lambda^i_{\;s\bar{j}},\qquad [N^i_{\;\bar{j}}]_{\bar{v}^l}=0.
$$

Comparing the expressions for $\Omega^i$ given by (\ref{eq:4_18}) with (\ref{eq:3_12}) yields
\begin{align*}
\mathcal{R}^i_{\;k\bar{j}}=-[N^i_{\;k}]_{v^l}\lambda^l_{\;s\bar{j}}v^s+[N^i_{\;k}]_{\bar{v}^s}N^{\bar{s}}_{\;\bar{j}}
  -N^i_{\;l}\lambda^l_{\;k\bar{j}}+\lambda^i_{\;l\bar{j}}N^l_{\;k}
  -\lambda^i_{\;s\bar{l}}v^s\lambda^{\bar{l}}_{\;\bar{j}k}.
\end{align*}
Using the flat connection $\mathcal D$ defined by \eqref{eq:4_17}, we can rewrite the curvature operator as
\begin{align*}
&\mathcal{R}(w,\sigma(w))v\nonumber\\
&=v^lw^k\bar{w}^j\mathcal{R}^i_{\;lk\bar{j}}\;\mathcal{E}_i=w^k\bar{w}^j \mathcal{R}^i_{\;k\bar{j}}\,\mathcal{E}_i \nonumber\\
&=w^k\bar{w}^j\Bigl\{-[N^i_{\;k}]_{v^l}\lambda^l_{\;s\bar{j}}v^s+[N^i_{\;k}]_{\bar{v}^s}N^{\bar{s}}_{\;\bar{j}}-N^i_{\;l}\lambda^l_{\;k\bar{j}}+\lambda^i_{\;l\bar{j}}N^l_{\;k}-\lambda^i_{\;s\bar{l}}v^s\lambda^{\bar{l}}_{\;\bar{j}k}\Bigr\}\,\mathcal{E}_i\nonumber\\
&=\mathcal{D}_{\sigma(\mathcal{N}_v(w))}\bigl(\mathcal{N}_v(w)\bigr)-\mathcal{D}_{[v,\sigma(w)]^{1,0}}\bigl(\mathcal{N}_v(w)\bigr)
  -\mathcal{N}_v\bigl([w,\sigma(w)]^{1,0}\bigr)+[\mathcal{N}_v(w),\sigma(w)]^{1,0}\nonumber\\
&\quad   -[v,[\sigma(w),w]^{0,1}]^{1,0}.
\end{align*}

Replacing $v$ by $w$ in the above formula and using $\mathcal{E}_i(N^j_{\;s}v^s)=\mathcal{E}_i(N^j_{\;s})v^s+N^j_{\;i}$ gives
\begin{align*}
\mathcal{R}(v,\sigma(v))v=&\mathcal{D}_{\sigma(\mathcal{N}_v(v))}\bigl(\mathcal{N}_v(v)\bigr)-\mathcal{D}_{[v,\sigma(v)]^{1,0}}\bigl(\mathcal{N}_v(v)\bigr)
            +[\mathcal{N}_v(v),\sigma(v)]^{1,0}\\
            &-[v,[\sigma(v),v]^{0,1}]^{1,0}.
\end{align*}

This completes the proof.
\end{proof}

\subsection{The case of complex Lie group}
\hspace{1em}
Note that a complex Lie group is naturally a real Lie group. In this section, we shall investigate left-invariant complex Finsler metrics on complex Lie groups.
\begin{definition}
	A complex Lie group $G$ is  a complex  manifold such that the group multiplication operation $(x, y)\to xy: G\times G\to G$ and
	 the inverse operation $x\to x^{-1}: G\to G$ are both holomorphic.
\end{definition}
As a complex manifold, there naturally exists a left-invariant complex structure $J$ on a complex Lie group $G$. Let $\mathfrak{g}$ be the real Lie algebra of $G$. Since $G$ is a complex Lie group, both left and right translations are holomorphic, which means $J$ is bi-invariant. It follows that the left-invariant complex vector fields $\iota_{\mathbb{C}}(X)$ are holomorphic for all $X\in \mathfrak{g}^{1,0}$. In particular, the global left-invariant frame $\{e_1,\dots,e_n\}$ of $T^{1,0}G$ constructed in Section \ref{LLL} is holomorphic.

Let $I$ be the endomorphism in $\mathfrak{g}$ corresponding to the left-invariant complex structure $J$ on $G$.
Since $G$ is a complex Lie group, the left-invariant complex structure $J$ is bi-invariant. Therefore, for any $X\in\mathfrak{g}$,
$$
\mathrm{ad}_X \circ I = I \circ \mathrm{ad}_X,
$$
where $ \mathrm{ad}_X(Y) = [X, Y]$ and  $Y\in\mathfrak{g} $.

\begin{proposition}\label{Po:4_14}
Let $G$ be a complex Lie group with real Lie algebra $\mathfrak{g}$, and let $J$ denote its natural left-invariant complex structure. Then for the decomposition of the complexification $\mathfrak{g}^\mathbb C = \mathfrak{g}^{1,0}\oplus \mathfrak{g}^{0,1}$ into the $\sqrt{-1}$ and $-\sqrt{-1}$ eigenspaces of $J$, we have
$$
[\mathfrak{g}^{1,0},\mathfrak{g}^{0,1}] = 0.
$$
\end{proposition}
\begin{proof}
For any $X\in\mathfrak{g}^{1,0}$ and $Y\in\mathfrak{g}^{0,1}$, since $IX = \sqrt{-1}X$ and $IY = -\sqrt{-1}Y$, evaluating $\operatorname{ad}_X(IY)=I(\operatorname{ad}_X Y)$ yields $-\sqrt{-1}[X,Y]=I[X,Y]$,
while $\operatorname{ad}_Y(IX)=I(\operatorname{ad}_Y X)$ gives $\sqrt{-1}[X,Y]=I[X,Y]$. Hence $[X,Y]=0$.
\end{proof}
It follows from Proposition \ref{Po:4_14} that the structure constants $\lambda^i_{\;j\bar k}$ and $\lambda^{\bar{i}}_{\;j\bar k}$ in (\ref{eq:4_1}) vanish identically,
so that the non-zero structure constants of $\mathfrak{g}^{\mathbb{C}}$ are $\lambda^i_{\;jk}$ and their conjugations.
\begin{theorem}\label{Th:4_15}
	Suppose $G$ is a complex Lie group and $F:T^{1,0}G\rightarrow[0,+\infty)$ is a left-invariant complex Finsler metric on $G$. Then
	$F$ must be a complex Berwald metric with vanishing holomorphic bisectional curvature.
\end{theorem}
\begin{proof}
By Proposition \ref{Po:4_14}, $\lambda^i_{\;j\bar k}=\lambda^{\bar{i}}_{\;j\bar k}=0$. Using \eqref{eq:4_16} yields
	\begin{align*}
		&N^j_{\;k}=\lambda_{o\;k}^{\;j}+g^{\bar{q}j}g_{\bar{q}\bar{l}}\lambda^{\bar{l}}_{\;\bar{o}k}
		=g^{\bar{l}j}g_{s\bar{i}}\lambda^{\bar{i}}_{\;\bar{l}k}v^s+g^{\bar{q}j}g_{\bar{q}\bar{l}}\lambda^{\bar{l}}_{\;\bar{s}k}\bar{v}^s
		=0.
	\end{align*}
	Using \eqref{eq:4_14}, we have
	\begin{align*}
		&\Gamma^j_{\;ik}=\lambda^{\ j}_{i\  k}-C^{\ j}_{i \ l}N^l_{\;k}+C^{\ j}_{i\ \bar s}\lambda^{\bar s}_{\ \bar ok}
		=g_{i\bar{s}}g^{\bar{l}j}\lambda^{\bar{s}}_{\;\bar{l}k}+C^{\ j}_{i\ \bar s}\lambda^{\bar s}_{\;\bar lk}\bar{v}^l=0.
	\end{align*}
	Thus, by Definition \ref{De:3_3}, $F$ is a complex Berwald metric on a complex Lie group $G$.

	The vanishing of the holomorphic bisectional curvature follows directly from Theorem \ref{Th:4_12}.
\end{proof}

\begin{corollary}
	Suppose $G$ is a complex Lie group and $F:T^{1,0}G\rightarrow[0,+\infty)$ is a left-invariant complex Finsler metric on $G$.
	Then $F$ is a K\"ahler-Berwald metric if and only if $G$ is an abelian  complex Lie group.
\end{corollary}
\begin{proof}
By \eqref{eq:4_15},  $T_{\;jk}^i=-\frac{1}{2}\lambda_{\; jk}^i$. On the other hand, by Theorem \ref{Th:4_15} and Definition \ref{De:3_5}, $F$ is a left-invariant K\"ahler-Berwald metric if and only if $T_{\;jk}^i=0$, namely $\lambda_{\; jk}^i=0$.

This completes the proof.
\end{proof}

\begin{remark}
The classification of connected compact complex Lie groups of complex dimension $n$ is very simple: every such group is isomorphic to a complex torus, namely the quotient $\mathbb{C}^n/\Gamma$, where $\Gamma$ is a discrete lattice of rank $2n$ in $\mathbb{C}^n$. The classification of  non-compact abelian complex Lie groups is more involved; according to Theorem 3.2 of Morimoto \cite{M}, every connected complex abelian Lie group $G$ is isomorphic to the direct product
$$G^0\times \mathbb{C}^m\times \mathbb{C}^{\ast n}(m,n\geq 0),$$
 where $G^0$ is a closed connected complex normal subgroup of $G$.
\end{remark}

\vskip0.4cm
\noindent
\textbf{Acknowledgements.}\ {\small C. Zhong was supported by  the National Natural Science Foundation of China (No. 12471080) and the Key Program of Fujian Provincial Natural Science Foundation of China (No.2026J002007)}.


\begin{thebibliography}{10}


\bibitem{AP}
M. Abate, G. Patrizio. {\em Finsler metrics---A Global Approach with Applications to Geometric Function Theory}, {\em Springer Lecture Notes in Math. 1591}, Springer, Berlin, 1994.

\bibitem{Aikou-b} T. Aikou. Complex manifolds modeled on a complex Minkowski space, \emph{J. Math. Kyoto Univ.} \textbf{35}(1): 85-103, 1995

\bibitem{BSS}
D. Bao, S.-S. Chern, Z. Shen. {\em An Introduction to Riemann-Finsler Geometry}, Vol. 200. Springer Science \& Business Media, 2012.

\bibitem{bao2002finsler}
D. Bao, Z. Shen. Finsler metrics of constant positive curvature on the Lie group $S^3$.  {\em J. Lond. Math. Soc.(2)}, \textbf{66}(2): 453-467,  2002.

\bibitem{Cao-Ge-Zhong} P. Cao, X. Ge, C. Zhong. Characterization of invariant complex Finsler metrics on the complex Grassmann manifold. {\em Differential Geom. Appl.} \textbf{94}, Paper No. 102138, 22pp, 2024.

\bibitem{ChenAndShen}
B. Chen, Y. Shen. K\"ahler Finsler metrics are actually strongly K\"ahler. {\em Chinese Ann. Math. Ser. B}, \textbf{30}(2): 173-178, 2009.


\bibitem{Deng-Hou1}
S. Deng,  Z. Hou. Invariant Finsler metrics on homogeneous manifolds. {\em J. Phys. A}, \textbf{37}(34): 8245-8253, 2004.

\bibitem{Deng-Hou2}
S. Deng,  Z. Hou. Invariant Finsler metrics on homogeneous manifolds. II. Complex structures. {\em J. Phys. A}, \textbf{39}(11): 2599-2609, 2006.


\bibitem{Deng1}
S. Deng. Homogeneous Finsler spaces. {\em Springer Monogr. Math.}, Springer, New York, xiv+240pp, 2012.

\bibitem{Ge-Zhong} X. Ge,  C. Zhong. Geometry of holomorphic invariant strongly pseudoconvex complex Finsler metrics on the classical domains. {\em Sci. China Math.} \textbf{67}(8): 1827-1864, 2024.


\bibitem{Heintze}
E. Heintze. On homogeneous manifolds of negative curvature. {\em Math. Ann.}, \textbf{211}: 23-34, 1974.


\bibitem{Huang1}
L. Huang.  Einstein Finsler metrics on $S^3$ with nonconstant flag curvature.  {\em Houston J. Math.}, \textbf{37}(4): 1071-1086, 2011.

\bibitem{Huang2}
L. Huang. Ricci curvatures of left invariant Finsler metrics on Lie groups.  {\em Israel J. Math.}, \textbf{207}(2): 783-792, 2015.

\bibitem{Huang-Mo}
L. Huang, X. Mo. Homogeneous Einstein Finsler metrics on $(4n+3)$-dimensional spheres.  {\em Canad. Math. Bull.}, \textbf{62}(3): 509-523, 2019.




\bibitem{FoundationOfDG2}
S. Kobayashi, K. Nomizu. {\em Foundations of Differential Geometry.} Vol. II.  Wiley-Interscience Publishers, New York-London, xv+470pp, 1969.

\bibitem{koszul1959exposes}
J. L. Koszul. Expos\'es sur les espaces homog$\grave{\mbox{e}}$nes sym\'etriques.  {\em  Sociedade de Matem\'atica de S$\tilde{\mbox{a}}$o Paulo, S$\tilde{\mbox{a}}$o Paulo}, 71pp, 1959.

\bibitem{Latifi-Toomanian}
D. Latifi, M. Toomanian. On the existence of bi-invariant Finsler metrics on Lie groups. {\em Math. Sci.}, \textbf{7}, Art. 37, 5pp, 2013.




\bibitem{Lin-Zhong}S. Lin,  C. Zhong. Characterization of holomorphic invariant strongly pseudoconvex complex Finsler metrics on unit polydisks. {\em J. Geom. Anal.} \textbf{33}(11), Paper No. 343, 28pp, 2023.


\bibitem{JMilnor}
J. Milnor. Curvatures of left invariant metrics on Lie groups. {\em Adv. Math.}, \textbf{21}(3): 293-329, 1976.

\bibitem{M}A. Morimoto. On the classification of noncompact complex abelian Lie groups. Trans. Amer. Math. Soc. \textbf{123}: 200-228, 1966.

\bibitem{SZ}
L. Sun, C. Zhong. Characterizations of complex Finsler connections and weakly complex Berwald metrics. Differential Geom. Appl. \textbf{31}(5): 648-671, 2013.

\bibitem{Vezzoni-Yang-Zheng}
L. Vezzoni, B. Yang, F. Zheng. Lie groups with flat Gauduchon connections. {\em Math. Z.}, \textbf{293}(1-2): 597-608, 2019.


\bibitem{Xu1}
M. Xu. Left invariant spray structure on a Lie group. {\em J. Lie Theory}, \textbf{32}(1): 121-138, 2022.

\bibitem{Xu2}
X. Xu, M. Xu. Left invariant complex Finsler metrics on a complex Lie group. arXiv: 2512.19353v1 [math.DG], 22 Dec 2025.


\bibitem{Zhong23}
C. Zhong. De Rham decomposition theorem for strongly convex K\"ahler-Berwald manifolds. {\em Results Math.}, \textbf{78}(1), Paper No. 25, 47pp, 2023.

\bibitem{Zhong25}
C. Zhong. Characterization of invariant complex Finsler metrics and Schwarz lemma on the classical domains. {\em Math. Ann.}, \textbf{392}(4): 4861-4890, 2025.

\end{thebibliography}
\end{document}